\documentclass[12pt]{article}
\title{Bulk scaling limit of the Laguerre ensemble}
\date{}
\author{St\' ephanie Jacquot\footnote{
University of Cambridge, Statistical Laboratory, Centre for Mathematical Sciences, Wilberforce Road, Cambridge, CB3 0WB, UK.   S.M.Jacquot@statslab.cam.ac.uk}\qquad Benedek Valk\'o\footnote{Department of Mathematics, University of Wisconsin - Madison, WI 53705, USA.  valko@math.wisc.edu B.Valk\'o was partially supported by the NSF Grant DMS-09-05820.}}

\usepackage{units}
\usepackage{amsfonts}
\usepackage{graphicx}
\usepackage{amsmath}
\usepackage{amsthm}
\usepackage{color}
    \newtheorem{theorem}{Theorem}
\newtheorem*{theorem*}{Theorem}
    \newtheorem{lemma}[theorem]{Lemma}
    \newtheorem{proposition}[theorem]{Proposition}
    \newtheorem{corollary}[theorem]{Corollary}

\theoremstyle{definition} 
    
    \newtheorem{remark}[theorem]{Remark}

\newcommand{\eps}{\varepsilon}
\newcommand{\Z}{{\mathbb Z}}
\newcommand{\ZZ}{{\mathbb Z}}
\newcommand{\FF}{{\mathcal F}}
\newcommand{\UU}{{\mathbb U}}
\newcommand{\R}{{\mathbb R}}
\newcommand{\CC}{{\mathbb C}}

\newcommand{\HH}{{\mathbb H} }
\newcommand{\ev}{{\mathbb   E}}
\newcommand{\pr}{\mbox{\rm P}}
\newcommand{\lstar}{{\raise-0.15ex\hbox{$\scriptstyle \ast$}}}
\newcommand{\lcirc}{{\raise-0.15ex\hbox{$\scriptscriptstyle \circ$}}}
\newcommand{\ldot}{.}
\newcommand{\vfi}{\varphi}
\theoremstyle{remark} 
\newcommand{\Sineb}{\operatorname{Sine}_{\beta}}
\newcommand{\Airyb}{\operatorname{Airy}_{\beta}}

\newcommand{\half}{\nicefrac{1}{2}\,}
\newcommand{\qt}{\nicefrac{1}{4}\,}
\newcommand{\tqt}{\nicefrac{3}{4}\,}
\newcommand{\gl}{\lambda}
\newcommand{\cp}{\stackrel{P}{\longrightarrow}}

\renewcommand{\O}{\mathcal{O}}
\definecolor{violet}{rgb}{0.8,0,0.2}
\newcommand{\cd}{\stackrel{d}{\Longrightarrow}}

\newcommand{\mbf}[1]{\mathbf{#1}}
\newcommand{\ash}{\operatorname{ash}}
\begin{document}
\maketitle
\begin{abstract}
We consider the $\beta$-Laguerre ensemble, a family of distributions generalizing the joint eigenvalue distribution of the Wishart random matrices. We show that  the  bulk scaling limit of  these ensembles  exists for all $\beta>0$ for a general family of parameters and it  is the same as the bulk scaling limit of the corresponding  $\beta$-Hermite ensemble.
\end{abstract}
\definecolor{Red}{rgb}{1,0,0}

\section{Introduction}
The Wishart-ensemble is one of the first studied random matrix models, introduced by Wishart in 1928 \cite{Wishart}. It describes the joint eigenvalue distribution of the $n\times n$ random symmetric matrix $M=A A^*$ where $A$ is an $n\times (m+1)$  matrix with i.i.d.~standard normal entries. One can also define versions with i.i.d.~complex or real quaternion standard normal random variables. Since we are only interested in the eigenvalues, we can assume $m+1\ge n$. Then the joint eigenvalue density on $\R_+^n$ exists and it is given by the following formula:
\begin{equation}\label{eq:Lag}
\frac{1}{Z_{n,m+1}^\beta}\prod_{j<k}
|\lambda_j-\lambda_k|^{\beta}\, \prod_{k=1}^n \lambda_k^{\frac{\beta}{2}(m-n)-1}  e^{-\frac{\beta}{2} \lambda_k}
\end{equation}
where $\beta=1,2$ and 4 correspond to the real, complex and quaternion cases respectively and $Z_{n,m+1}^\beta$ is an explicitly computable constant.

 The density (\ref{eq:Lag}) defines a distribution on $\R_+^n$  for any $\beta>0$, $n\in \mathbb{N}$ and $m>n$. The resulting family of distributions is called the $\beta$-Laguerre ensemble. Note that we intentionally shifted the parameter $m$  by one as this will result in slightly cleaner expressions later on.

Another important family of distributions in random matrix theory is the Hermite (or Gaussian) $\beta$-ensemble. It is described by the density function
\begin{eqnarray}\label{eq:Herm}
\frac{1}{\tilde Z_n^\beta}\prod_{1\le j<k \le n}
|\lambda_j-\lambda_k|^{\beta}\, \prod_{k=1}^n e^{-\frac{\beta}{4} \lambda_k^2}.
\end{eqnarray}
on $\R^n$.
For $\beta=1,2$ and $4$ this gives the joint eigenvalue density of the Gaussian orthogonal, unitary and symplectic ensembles. It is known that if we rescale the ensemble by $\sqrt{n}$ then the empirical spectral density converges to the Wigner semicircle distribution $\frac{1}{2\pi}\sqrt{4-x^2} 1_{[-2,2]}$.
 In \cite{carousel} the authors  derive the bulk scaling limit of the $\beta$-Hermite ensemble, i.e.~the point process limit of the spectrum it is  scaled around a sequence of points away from the edges. 
\begin{theorem}[Valk\'o and Vir\' ag \cite{carousel}]\label{thm:carousel}
 If $\mu_n$ satisfies $n^{1/6}(2\sqrt{n}-|\mu_n|)\to \infty$ as $n\to \infty$  and $\Lambda_n^{H}$ is a sequence of random vectors with density (\ref{eq:Herm}) then
\begin{equation}
\sqrt{4n-\mu_n^2} (\Lambda_n^H-\mu_n)\Rightarrow \Sineb
\end{equation}
where $\Sineb$ is a discrete point process with density $(2\pi)^{-1}$.
\end{theorem}
Note that the condition on $\mu_n$  means that we are in the bulk of the spectrum, not too close to the edge.
 The limiting point process $\Sineb$ can be described as a functional  of the Brownian motion in the hyperbolic plane or equivalently via a system of stochastic differential equations (see Subsection \ref{subs:sineb} for details).

The main result of the present paper provides the point process limit of the Laguerre ensemble in the bulk. In order to understand the order of the scaling parameters, we first recall the classical results about the limit of the empirical spectral measure for the Wishart matrices.  If  $m/n\to \gamma \in[1,\infty)$ then with probability one the scaled empirical spectral measures $\nu_n=\frac1n \sum_{k=1}^n \delta_{\lambda_k/n}$ converge weakly to the Marchenko-Pastur distribution which is a deterministic measure with density
\begin{equation}\label{eq:MP}
\tilde \sigma^\gamma(x)=\frac{\sqrt{(x-a^2)(b^2-x)}}{2\pi x } 1_{[a^2,b^2]}(x), \qquad a=a(\gamma)=\gamma^{1/2}-1, \,\,b=b(\gamma)=1+\gamma^{1/2}.
\end{equation}
This can be proved  by the moment method or using Stieltjes-transform. (See \cite{MP} for the original proof and \cite{ForBook} for the general $\beta$ case). 

Now we are ready to state our main theorem:
\begin{theorem}[Bulk limit of the Laguerre ensemble]\label{thm:main}
Fix $\beta>0$, assume that $m/n\to \gamma \in [1,\infty)$ and let $ c \in (a^2, b^2)$. Let $\Lambda_n^L$ denote the point process given by (\ref{eq:Lag}). Then
\begin{equation}
2\pi \tilde \sigma^\gamma(c) \left(\Lambda_n^L-c n\right)\Rightarrow \Sineb
\end{equation}
where $\Sineb$ is the bulk scaling limit of the $\beta$-Hermite ensemble.
\end{theorem}
We will actually prove a more general version of this theorem: we will also allow the cases when $m/n\to \infty$ or when the center of the scaling  gets close to the spectral edge. See Theorem \ref{thm:second} in Subsection \ref{subs:thm} for the details.

Although this statement has been known for the classical cases ($\beta=1,2$ and 4) \cite{mehta}, this is the first proof for general $\beta$.  Our approach relies on the tridiagonal matrix representation of the Laguerre ensemble introduced by Dumitriu and Edelman \cite{DE} and the techniques introduced in \cite{carousel}.  

There are various other ways one can generalize the classical Wishart ensembles. One possibility is that instead of  normal random variables one uses  more general distributions
 in the construction described at the beginning of this section. The recent papers of Tao and Vu \cite{TV_wishart} and Erd\H{o}s et al.~\cite{ESYY_wishart} provide the bulk scaling limit in these cases.

Our theorem completes the picture about the point process scaling limits of the Laguerre ensemble. The scaling limit at the soft edge has been proved in \cite{RRV}, where the edge limit of the Hermite ensemble was also treated.
\begin{theorem}[Ram\'\i rez, Rider and Vir\'ag \cite{RRV}]\label{thm:softedge}
If $m>n\to \infty$
then
\[
\frac{(mn)^{1/6}}{(\sqrt{m}+\sqrt{n})^{4/3}}(\Lambda_n^L-(\sqrt{n}+\sqrt{m})^2)\Rightarrow \Airyb
\]
where $\Airyb$ is a discrete simple point process given by the eigenvalues of the stochastic Airy operator
\[
\mathcal H_\beta=-\frac{d^2}{dx^2}+x+\frac{2}{\sqrt{\beta}} b_x'.
\]
Here $b'_x$ is white noise and the eigenvalue problem is set up on the positive half line with initial conditions $f(0)=0, f'(0)=1$.

A similar limit holds at the lower edge: if $\liminf m/n>1$ then
\[
\frac{(mn)^{1/6}}{(\sqrt{m}-\sqrt{n})^{4/3}}((\sqrt{m}-\sqrt{n})^2-\Lambda_n^L)\Rightarrow \Airyb.
\]
\end{theorem}
\begin{remark}
The lower edge result is not stated explicitly in \cite{RRV}, but it follows by a straightforward modification of the proof of the upper edge statement. Note that the condition $\liminf m/n>1$ is not optimal, the statement is expected to hold with $m-n\to \infty$. This has been known for the classical cases $\beta=1,2,4$ \cite{mehta}.
\end{remark}
If $m-n\to a\in (0,\infty)$ then the lower edge of the spectrum is pushed to 0 and it becomes a `hard' edge. The scaling limit in this case was proved in \cite{RR}.
\begin{theorem}[Ram\'\i rez and Rider \cite{RR}]
If $m-n\to a \in (0,\infty)$ then
\[
n \Lambda_n^L\Rightarrow \Theta_{\beta,a}
\]
where $\Theta_{\beta,a}$ is a simple point process that can be described as the sequence  of eigenvalues of a certain random operator.
\end{theorem}

%

In the next section we discuss the tridiagonal representation of the Laguerre ensemble, recall how to count eigenvalues of a tridiagonal matrix and state a more general version of our theorem. Section \ref{s:proof} will contain the outline of the proof while the rest of the paper deals with the details of the proof.

\section{Preparatory steps}

\subsection{Tridiagonal representation}

In \cite{DE} Dumitriu and Edelman proved that the $\beta$-Laguerre ensemble can be represented as joint eigenvalue distributions for certain random tridiagonal matrices.
Let $A_{n,m}$ be the following $n\times n$ bidiagonal matrix:
\[
A_{n,m}=\frac{1}{\sqrt{\beta}}\left[
\begin{array}{ccccc}
\tilde \chi_{\beta (m-1)}&&&&\\
\chi_{\beta(n-1)}&\tilde \chi_{\beta(m-2)}&&&\\
&\ddots&\ddots&&\\
&&\chi_{\beta\cdot 2}&\tilde \chi_{\beta (m-n+1)}&\\
&&&\chi_\beta&\tilde \chi_{\beta(m-n)}
\end{array}
     \right].
\]
where $\chi_{\beta a}, \tilde \chi_{\beta b}$  are independent chi-distributed random variables with the appropriate parameters ($1\le a\le n-1, m-1\le b\le m-n$).
Then the eigenvalues of the tridiagonal matrix $A_{n,m} A_{n,m}^T$ are distributed according to the density (\ref{eq:Lag}).

If we want to find the bulk scaling limit of the eigenvalues of $A_{n,m} A_{n,m}^T$ then it is sufficient to understand the scaling limit of the singular values of $A_{n,m}$.The   following simple lemma will be a useful tool for this.
\begin{lemma}
Suppose that $B$ is an $n \times n$ bidiagonal matrix with $a_1, a_2, \dots, a_n$ in the diagonal and $b_1, b_2, \dots, b_{n-1}$ below the diagonal. Consider the $2n\times 2n$ symmetric tridiagonal matrix
$M$ which has zeros in the main diagonal and $a_1, b_1, a_2, b_2, \dots, a_n$ in the off-diagonal. If the singular values of $B$ are $\lambda_1, \lambda_2, \dots, \lambda_n$ then the eigenvalues of $M$ are $\pm \lambda_i, i=1\dots n$.
\end{lemma}
\noindent We learned about this trick from  \cite{DF}, we reproduce the simple proof for the sake of the reader.
\begin{proof}
Consider the matrix $\tilde B=\left[\begin{array}{cc}
0&B^T\\B&0
\end{array}\right]$. If $Au=\lambda_i v$ and $A^T v=\lambda_i u$ then $[u,\pm v]^T$ is an eigenvector of $\tilde B$ with eigenvalue $\pm \lambda_i$.
 Let $C$ be the permutation matrix corresponding to $(2,4,\dots, 2n, 1,3,\dots,2n-1)$. Then  $C^T \tilde B C$ is exactly the tridiagonal matrix described in the lemma and its eigenvalues are exactly $\pm \lambda_i, i=1\dots n$.
\end{proof}
Because of the previous lemma it is enough to study the eigenvalues of the $(2n)\times(2n)$ tridiagonal matrix
\begin{equation}\label{eq:matrix}
\tilde A_{n,m}=\frac{1}{\sqrt{\beta}}\left[
\begin{array}{ccccccc}
0&\tilde \chi_{\beta (m-1)}&&&&\\
\tilde \chi_{\beta (m-1)}&0&\chi_{\beta(n-1)}&&&\\
&\chi_{\beta(n-1)}&0&\tilde \chi_{\beta (m-2)}&&\\
&&\ddots&\ddots&\ddots&\\
&&&\tilde \chi_{\beta (m-n+1)}&0&\chi_{\beta}\\
&&&&\chi_{\beta}&0&{\tilde \chi_{\beta (m-n)}}\\
&&&&&{\tilde \chi_{\beta (m-n)}}&{0}
\end{array}
     \right]
\end{equation}
The main advantage of this representation, as opposed to studying the tridiagonal matrix $A_{n,m} A_{n,m}^T$, is that here the entries are independent modulo symmetry.
\begin{remark}\label{rem:sing}
Assume that $[u_1,v_1,u_2,v_2,\dots, u_n,v_n]^T$ is an eigenvector for $\tilde A_{n,m}$ with eigenvalue $\lambda$. Then $[u_1,u_2,\dots,u_n]^T$ is and eigenvector for $A_{n,m}^T A_{n,m}$ with eigenvalue $\lambda^2$ and $[v_1,v_2,\dots,v_n]^T$ is an eigenvector for $A_{n,m} A_{n,m}^T$ with eigenvalue $\lambda^2$.
\end{remark}

\subsection{Bulk limit of the singular values}\label{subs:thm}

We can compute the asymptotic spectral density of $\tilde A_{n,m}$ from the Marchenko-Pastur distribution. If  $m/n\to \gamma\in [1,\infty)$ then the asymptotic density (when scaled with $\sqrt{n}$) is
\begin{eqnarray}\notag
\sigma^\gamma(x)&=&2|x| \tilde \sigma^\gamma (x^2)= \frac{\sqrt{(x^2-a^2)(b^2-x^2)}}{\pi x } 1_{[a,b]}(|x|)\\&=& \frac{\sqrt{(x-a)(x+a)(b-x)(b+x)}}{\pi x } 1_{[a,b]}(|x|).\label{eq:sp_dens}
\end{eqnarray}
This means that the spectrum of $\tilde A_{n,m}$ in $\R^+$ is asymptotically concentrated on the interval $[\sqrt{m}-\sqrt{n}, \sqrt{m}+\sqrt{n}]$. We will  scale around $\mu_n\in (\sqrt{m}-\sqrt{n}, \sqrt{m}+\sqrt{n})$ where  $\mu_n$ is chosen  in a way  that it is not too close to the edges. Near $\mu_n$ the asymptotic eigenvalue density should be close to $\sigma^{m/n}(\mu_n/\sqrt{n})$ which explains the choice of the scaling parameters in the following theorem.
\begin{theorem}\label{thm:second}
Fix $\beta>0$ and suppose that $m=m(n)> n$. Let $\Lambda_n$ denote the set of eigenvalues of $\tilde A_{n,m}$ and set
\begin{eqnarray}\label{eq:n0}
n_0&=&\frac{\pi^2}{4} n \,\sigma^{m/n}\!\!\left({\mu_n} n^{-1/2}\right)^2-\frac12, \qquad n_1=n-\frac{\pi^2}{4} n \,\sigma^{m/n}\!\!\left({\mu_n} n^{-1/2}\right)^2.
\end{eqnarray}
Assume that as $n\to \infty$ we have
\begin{equation}\label{eq:n1_cond}
n_1^{1/3} n_0^{-1}\to 0
\end{equation}
and
\begin{equation}\label{eq:mu_cond}
\liminf_{n\to \infty} m/n>1 \qquad \textup{or} \qquad \lim_{n\to \infty} m/n=1 \textup{ and } \liminf \mu_n/\sqrt{n}>0.
\end{equation}
Then 
\begin{eqnarray}\label{eq:limit1}
 4\sqrt{n_0}  (\Lambda_n-\mu_n)\Rightarrow \Sineb.
\end{eqnarray}
\end{theorem}
The extra 1/2 in the definition of $n_0$ is introduced to make some of the forthcoming formulas nicer.
We also note that the following identities hold:
\begin{align}
n_0+\frac12=\frac{2 (m+n) \mu_n^2-(m-n)^2-\mu_n^4}{4 \mu_n^2},
\qquad n_1=\frac{\left(m-n-\mu_n^2\right)^2}{4 \mu_n ^2}.\label{eq:n0n1}
\end{align}


Note that we did not assume that $m/n$ converges to a constant or that $\mu_n=c \sqrt{n}$. By the discussions at the beginning of this section $(\Lambda_n\cap \R^+)^2$ is distributed according to the Laguerre ensemble. If we assume that $m/n\to \gamma$ and $\mu_n=\sqrt{c} \sqrt{n}$ with $c\in (a(\gamma)^2,b(\gamma)^2)$ then both (\ref{eq:n1_cond}) and (\ref{eq:mu_cond}) are satisfied. Since in this case $n_0 n^{-1}\to \tilde \sigma^\gamma(c)$ the
 result of  Theorem \ref{thm:second} implies Theorem \ref{thm:main}.
%
%

\begin{remark}\label{rem:subseq}
We want prove that the weak limit of $4\sqrt{n_0}(\Lambda_n-\mu_n)$ is $\Sineb$, thus it is sufficient to prove that for any subsequence of $n$ there is a further subsequence so that the  limit in distribution holds. Because of this by taking an appropriate subsequence we may assume that
\begin{eqnarray}\label{extracond}
m/n\to \gamma \in [1,\infty],\qquad \textup{and if} \quad m/n\to 1 \quad \textup{then } \quad \mu_n/\sqrt{n}\to c\in(0,2].
\end{eqnarray}
These assumptions imply that for $m_1=m-n+n_1$ we have
\begin{eqnarray}\label{m1}
\liminf m_1/n>0.
\end{eqnarray}
One only needs to check this in the $m/n\to 1$ case, when from (\ref{extracond}) and the definition of $n_1$ we get $n_1/n\to c>0$.
\end{remark}

\begin{remark}
The conditions of Theorem \ref{thm:second} are optimal if $\liminf m/n>1$ and the theorem
provides a complete description of the possible point process scaling limits of $\Lambda^L_n$. To see this first note that using $\Lambda_n^L=(\Lambda_n\cap \R^+)^2$ we can translate the edge scaling limit of Theorem \ref{thm:softedge} to  get
\begin{equation}\label{eq:edge}
\frac{2(mn)^{1/6}}{(\sqrt{m}\pm\sqrt{n})^{1/3}}(\Lambda_n-(\sqrt{m}\pm\sqrt{n}))\Rightarrow \pm \Airyb.
\end{equation}
If $\liminf m/n>1$ then by the previous remark we may assume $\lim m/n=\gamma\in (1,\infty]$. Then the previous statement can be transformed into $n^{1/6}(\Lambda_n-(\sqrt{m}\pm\sqrt{n}))\cd \Xi$ where $\Xi$ is a 
 a linear transformation of $\Airyb$. From this it is easy to check that if $n_1^{1/3} n_0^{-1}\to c \in (0,\infty]$ then we need to scale $\Lambda_n-\mu_n$ with $n^{1/6}$ to get a meaningful limit (and the limit is a linear transformation of $\Airyb$) and if $n_1^{1/3} n_0^{-1}\to 0$ then we get the bulk case.

If $m/n\to 1$ then the condition (\ref{eq:mu_cond}) is suboptimal, this is partly  due to the fact that the lower soft edge limit in this case is not available.
Here the statement should be true with the following condition instead of (\ref{eq:mu_cond}): 
\begin{eqnarray*}
&&\mu_n{\sqrt{n}}{(m-n)^{-1/3}}-\frac12 (m-n)^{2/3}\to \infty.
\end{eqnarray*}
\end{remark}

%
%
%
%

\subsection{Counting eigenvalues of tridiagonal matrices}\label{subs:evcount}

Assume that the tridiagonal $k\times k$ matrix $M$ has positive off-diagonal entries.
\[
M=\left[
\begin{array}{ccccc}
a_1&b_1&&\\
c_1&a_2&b_2&\\
&\ddots&\ddots&\\
&&c_{k-2}&a_{k-1}&b_{k-1}\\
&&&c_{k-1}&a_k
\end{array}\right], \qquad b_i>0, c_i>0.
\]
If $u=\left[u_1,\dots,u_k\right]^T$ is an eigenvector corresponding to $\lambda$ then we have
\begin{equation}
c_{\ell-1} u_{\ell-1}+a_{\ell} u_{\ell}+b_{\ell} u_{\ell+1}=\lambda u_\ell, \qquad \ell=1,\dots k\label{ev}
\end{equation}
where we can we set $u_0=u_{k+1}=0$ (with $c_0, b_k$ defined arbitrarily). This gives a single term recursion on $\R\cup\{\infty\}$ for the ratios $r_\ell=\frac{u_{\ell+1}}{u_\ell}$:
\begin{equation}\label{r_rec}
r_0=\infty, \quad r_{\ell}=\frac{1}{b_\ell}\left(-\frac{c_{\ell-1}}{r_{\ell-1}}+\lambda-a_\ell
\right), \qquad \ell=1,\dots k.
\end{equation}
This recursion can be solved for any parameter $\lambda$, and $\lambda$ is an eigenvalue if and only if $r_k=r_{k,\lambda}=0$.

We can turn $r_{\ell,\lambda}$ into an angle $\phi_{\ell,\lambda}$ with $\lambda\to \phi_{\ell,\lambda}$ being continuous, monotone increasing and $2\tan(\phi)=r$.
Then  the values $\phi_{k,\lambda_0}$ and $\phi_{k,\lambda_1}$ will determine the number of eigenvalues of $M$ in $[\lambda_0,\lambda_1]$:
\[
\# \left\{ (\phi_{k,\lambda_0},\phi_{k, \lambda_1}]\cap 2\pi \Z \right\}=\#\{\textup{eigenvalues in $(\lambda_0, \lambda_1]$}\}
\]
This is basically a discrete version of the Sturm-Liouville oscillation theory.

We do not need to fully solve the recursion  (\ref{r_rec}) in order to count eigenvalues. If we consider the reversed version of (\ref{r_rec}) started from index $k$ with initial condition $0$:
\begin{equation}
r^{\odot}_k=0, \quad r^{\odot}_{\ell-1}=-c_{\ell}\left({a_{\ell}}-\lambda+b_\ell r^{\odot}_\ell\right)^{-1}, \qquad \ell=1,\dots k.
\end{equation}
then  $\lambda$ is an eigenvalue if and only if $r_{\ell,\lambda}=r^{\odot}_{k-\ell,\lambda}$. Moreover, if we turn $r^{\odot}_{\ell,\lambda}$ into an angle $\phi^{\odot}_{\ell, \lambda}$ (similarly as before for $r$ and $\phi$) we can also count eigenvalues in the interval $[\lambda_0, \lambda_1]$ by the formula
\begin{equation}\label{evcount}
\# \left\{ (\phi_{\ell,\lambda_0}-\phi^{\odot}_{\ell,\lambda_0},\phi_{\ell,\lambda_1}-\phi^{\odot}_{\ell,\lambda_1}]\cap 2\pi \Z \right\}=\#\{\textup{eigenvalues in $(\lambda_0, \lambda_1]$}\}
\end{equation}

In our case, by analyzing the scaling limit of a certain version of the phase function $\phi_{\ell,\lambda}$ we can identify the limiting point process.  This method was used in \cite{carousel} for the bulk scaling limit of the $\beta$ Hermite ensemble. An equivalent approach (via transfer matrices) was used in  \cite{KVV} and \cite{VV3} to analyze the asymptotic behavior of the spectrum for certain discrete random Schr\"odinger operators.

\subsection{The $\Sineb$ process}\label{subs:sineb}


The distribution of the point process $\Sineb$ from Theorem \ref{thm:carousel} was described in \cite{carousel} as a functional of the Brownian motion in the hyperbolic plane (the Brownian carousel) or equivalently via a system of stochastic differential equations. We review the latter description here. Let $Z$ be a complex Brownian motion with i.i.d.~standard real and imaginary parts. Consider the strong solution of the following one parameter system of stochastic differential equations for $t\in [0,1)$, $\lambda\in \R$\,:\begin{equation}\label{sineb}
d\alpha_\lambda= \frac{\lambda}{2 \sqrt{1-t}} dt+\frac{\sqrt{2}}{\sqrt{\beta(1-t)}}\Re\left[(e^{-i \alpha_\lambda}-1)dZ\right], \qquad \alpha_\lambda(0)=0.
\end{equation}
It was proved in \cite{carousel} that for any given $\lambda$ the limit  $N(\lambda)=\frac1{2\pi} \lim_{t\to 1} \alpha_{\lambda}(t)$ exists, it is integer valued a.s.
and $N(\lambda)$ has the same distribution as the counting function  of the point process $\Sineb$ evaluated at $\lambda$. Moreover, this is true for the joint distribution of $(N(\lambda_i), i=1,\dots,d)$ for any fixed vector $(\lambda_i,i=1,\dots,d)$. Recall that the counting function at $\lambda>0$ gives the number of points in the interval $(0,\lambda]$, and negative the number of points in $(\lambda,0]$ for $\lambda<0$.

%
%

%
%
%
%
%
%

\section{The main steps of the proof of Theorem \ref{thm:second}}\label{s:proof}

The proof will be similar to one given for Theorem \ref{thm:carousel} in \cite{carousel}. The basic 
idea is simple to explain: we will define a version of the phase function and the target phase function for the rescaled eigenvalue equation and consider (\ref{evcount}) with a certain $\ell=\ell(n)$. We will then show that the length of the interval in the left hand side of the equation converges to $2\pi(N(\lambda_1)-N(\lambda_0))$ while the left endpoint of that interval becomes uniform modulo $2\pi$. This shows that the scaling limit of the eigenvalue process is given by $\Sineb$.

The actual proof will require several steps. In order to limit the size of this paper and not to make it overly technical, we will recycle some  parts of the proof in \cite{carousel}. Our aim is to give full details whenever there is a major difference between the two proofs and to provide an outline of the proof if one can adapt parts of \cite{carousel} easily.

\begin{proof}[Proof of Theorem \ref{thm:second}]

Recall that $\Lambda_n$ denotes the multi-set of eigenvalues for the matrix $\tilde A_{n,m}$ which is  defined in (\ref{eq:matrix}).  We denote by $N_n(\lambda)$ the counting function of the scaled random multi-sets $4 n_0^{1/2}(\Lambda_n-\mu_n)$, we will prove that for any $(\lambda_1,\cdots,\lambda_d)\in\mathbb{R}^d$ we have 
\begin{equation}\label{Nlim}
\left(N_n(\lambda_1),\cdots,N_n(\lambda_d)\right)
\stackrel{d}{\Longrightarrow}
\left(N(\lambda_1),\cdots,N(\lambda_d)\right).
\end{equation}
where $N(\lambda)=\frac1{2\pi}\lim_{t\to 1} \alpha_{\lambda}(t)$ as defined using the SDE (\ref{sineb}).

We will use the ideas described in Subsection \ref{subs:evcount} to analyze the eigenvalue equation  $\tilde A_{n,m} {x}=\Lambda{x}$, where ${x}\in \R^{2n}$. Following the scaling given in (\ref{eq:limit1}) we set
\[\Lambda=\mu_n+\frac{\lambda}{4\sqrt{n_0}}.\]
In Section \ref{s:phase} we will define the phase function $\vfi_{\ell, \lambda}$ and the target phase function $\vfi^{\odot}_{\ell, \lambda}$ for $\ell\in[0,n_0)$. These will be independent of each other  for a fixed $\ell$ (as functions in $\lambda$) and satisfy the following identity for $\lambda<\lambda'$:
\begin{equation}\label{evcount1}
\# \left\{ (\vfi_{\ell,\lambda}-\vfi^{\odot}_{\ell,\lambda},\vfi_{\ell,\lambda'}-\vfi^{\odot}_{\ell,\lambda'}]\cap 2\pi \Z \right\}=N_n(\lambda')-N_n(\lambda).
\end{equation}
The phase function $\vfi$ will be a regularized version of the phase function obtained from the ratio of the consecutive elements of the eigenvector. The regularization is needed in order to have a phase function which is asymptotically continuous. Indeed, in Proposition \ref{prop:phisde} of Section \ref{s:sde} we will show that for any $0<\eps<1$  the rescaled version of the phase function $\vfi_{\ell, \lambda}$ in $\left[0, n_0(1-\eps)\right]$ converges to a one-parameter family of stochastic differential equations. Moreover we will prove that in the same region the relative phase function $\alpha_{\ell, \lambda}=\vfi_{\ell, \lambda}-\vfi_{\ell, 0}$ will converge to the solution $\alpha_\lambda$ of the SDE (\ref{sineb})\begin{equation}
\alpha_{\lfloor n_0(1-\eps)\rfloor, \lambda}\cd \alpha_{\lambda}(1-\eps), \quad \textup{as $n\to \infty$}
\end{equation}
in the sense of finite dimensional distributions in $\lambda$. This will be the content of Corollary \ref{cor:sineb}.

Next we will describe the asymptotic behavior of the phase functions $\vfi_{\ell, \lambda}, \alpha_{\ell, \lambda}$ and $\vfi^{\odot}_{\ell, \lambda}$ in the stretch $\ell \in [\lfloor n_0(1-\eps)\rfloor, n_2]$ where
\begin{equation}
n_2=\lfloor n_0-\mathcal{K}(n_1^{1/3}\vee 1)\rfloor.
\end{equation}
(The constants $\eps, \mathcal{K}$ will be determined later.)
We will show that if the relative phase function is already close to an integer multiple of $2\pi$ at $\lfloor n_0(1-\eps)\rfloor$ then it will not change too much in the interval $[\lfloor n_0(1-\eps)\rfloor,n_2]$. To be more precise, in Proposition \ref{prop:middle} of Section \ref{s:middle} we will prove that
there exists a constant $c=c(\bar\lambda,\beta)$ so that  we have
\[\mathbb{E}\left[|(\alpha_{\lfloor n_0(1-\eps)\rfloor,\lambda}-\alpha_{n_2,\lambda})\wedge 1\right]\leq c
\left[\textup{dist}(\alpha_{\lfloor n_0(1-\eps)\rfloor,\lambda},2\pi\mathbb{Z})+\sqrt{\epsilon}+n_0^{-1/2}(n_1^{1/6}{\vee} \log n_0)+\mathcal{K}^{-1}\right]\]
for all $\mathcal{K}>0, \epsilon\in(0,1),\lambda\leq |\bar\lambda|$.

We will also show that if  $\mathcal{K}\to \infty$ and $\mathcal{K}(n_1^{1/3}\vee 1) n_0^{-1}\to 0$ then  the random angle $\varphi_{n_2,0}$ becomes uniformly distributed module $2\pi$ as $n\to \infty$ (see Proposition \ref{prop:unifmiddle}).
 
Next we will prove that the target phase function will loose its dependence on $\lambda$: for every $\lambda\in \R$ and $\mathcal{K}>0$ we have
\begin{equation}
\alpha^{\odot}_{\ell, \lambda}=\varphi^{\odot}_{n_2,\lambda}-\varphi^{\odot}_{n_2,0}\cp 0, \quad \textup{as $n\to \infty$}.
\end{equation}
This will be the content of Proposition \ref{prop:end} in Section \ref{s:last}.

The proof now can be finished exactly the same way as in \cite{carousel}. We can choose $\eps=\eps(n)\to 0$ and $\mathcal{K}=\mathcal{K}(n)\to \infty$ so that  the following limits all hold simultaneously:
\begin{eqnarray*}
(\alpha_{\lfloor n_0(1-\eps)\rfloor, \lambda_i}, i=1\dots d)&\cd& (2\pi N(\lambda_i), i=1\dots d),\\
\vfi_{n_2, 0}&\cp& \textup{Uniform}[0,2\pi]\quad \textup{modulo $2\pi$},\\
\alpha_{\lfloor n_0(1-\eps)\rfloor, \lambda_i}-\alpha_{n_2, \lambda_i}&\cp& 0, \quad i=1,\dots,d,\\
\alpha^{\odot}_{n_2,\lambda_i}&\cp& 0, \quad i=1,\dots,d.
\end{eqnarray*}
This means that if we apply the identity (\ref{evcount1}) with $\lambda=0, \lambda'=\lambda_i$ and $\ell=n_2$ then the length of the random intervals 
\[
I_i=
(\vfi_{n_2,0}-\vfi^{\odot}_{n_2,0},\vfi_{n_2,\lambda_i}-\vfi^{\odot}_{n_2,\lambda_i}]\
\]
converge to $2\pi N(\lambda_i)$ in distribution (jointly), while the common left 
endpoint of these intervals becomes uniform modulo $2\pi$. (Since $\vfi_{n_2,0}$ and $\vfi^{\odot}_{n_2,0}$ are independent and $\vfi_{n_2,0}$ converges to a uniform distribution mod $2\pi$.)
This means that $\#\{2 k\pi \in I_i: k\in \Z\}$    converges to $N(\lambda_i)$ which proves (\ref{Nlim}) and Theorem \ref{thm:second}. 
\end{proof}

\section{Phase functions}\label{s:phase}

In this section we introduce the phase functions used to count the eigenvalues.

\subsection{The eigenvalue equations}

Let $s_j=\sqrt{n-j-1/2}$ and $p_j=\sqrt{m-j-1/2}$. Conjugating the matrix $\tilde A_{n,m}$ (\ref{eq:matrix}) with a $(2n)\times(2n)$ diagonal matrix $D=D^{(n)}$ with diagonal elements

\[
D_{2i,2i}=\prod_{\ell=1}^i \frac{\tilde \chi_{\beta (m-\ell)} \chi_{\beta(n-\ell)}}{\beta p_\ell s_\ell}, \qquad D_{2i+1,2i+1}=\frac{\tilde \chi_{\beta (m-i-1)}}{\sqrt{\beta} p_{i+1}}  \prod_{\ell=1}^i \frac{\tilde \chi_{\beta (m-\ell)} \chi_{\beta(n-\ell)}}{\beta p_{\ell} s_\ell}
\]
we get the tridiagonal matrix $\tilde A_{n,m}^D=D^{-1} \tilde A_{n,m} D$:
\begin{equation}\label{eq:tridag}
\tilde A_{n,m}^D=\left[
\begin{array}{ccccccc}
0&p_0+X_0&&&&\\
p_1&0&s_0+Y_0&&&\\
&s_1&0&p_1+X_1&&\\
&&\ddots&\ddots&\ddots&\\
&&&p_{n-1}&0&s_{n-2}+Y_{n-2}\\
&&&&s_{n-1}&0&{p_{n-1}+X_{n-1}}\\
&&&&&{p_{n}}&{0}
\end{array}
     \right]
\end{equation}
where
\[
X_{\ell}=\frac{\tilde \chi_{\beta(m-\ell-1)}^{2}}{\beta p_{\ell+1}}-p_{\ell}, \quad 0\leq \ell\leq n-1
, \qquad
Y_{\ell}=\frac{\chi_{\beta(n-\ell-1)}^{2}}{\beta s_{\ell+1}}-s_{\ell}, \quad 0\leq \ell\leq n-2.
\]
The first couple of moments of these random variables are explicitly computable using the moment generating function of the $\chi^2$-distribution and we get the following asymptotics:
\begin{equation}
\begin{array}{c}
\ev X_{\ell}=\O((m-\ell)^{-3/2}),\quad  \ev X_\ell^2=2/\beta+\O((m-\ell)^{-1}), \quad \ev X_{\ell}^4=\O(1),\\[4pt]
\ev Y_{\ell}=\O((n-\ell)^{-3/2}),\quad  \ev Y_\ell^2=2/\beta+\O((n-\ell)^{-1}), \quad \ev Y_{\ell}^4=\O(1),
\end{array}\label{eq:moments}
\end{equation}
where the constants in the error terms only depend on $\beta$.

We consider the eigenvalue equation for $\tilde A_{n,m}^D$ with a given $\Lambda\in \R$ and denote a nontrivial solution of the  first $2n-1$ components by  $u_1,v_1,u_2, v_2,\dots, u_n, v_n$. Then we have
\begin{eqnarray*}
s_\ell v_\ell+(p_{\ell}+{X_{\ell}})v_{\ell+1}&=&\Lambda u_{\ell+1},\qquad 0\le \ell\le n-1,\\
p_{\ell+1} u_{\ell+1}+(s_{\ell}+{Y_{\ell}}) u_{\ell+2}&=&\Lambda v_{\ell+1}, \qquad 0\le \ell \le n-2,
\end{eqnarray*}
where we set $v_0=0$ and we can assume $u_1=1$ by linearity.
We set  $r_\ell=r_{\ell,\Lambda}=u_{\ell+1}/v_{\ell}$, $0\le \ell\le n-1$ and $\hat r_\ell=\hat r_{\ell,\Lambda}=v_\ell/u_\ell$, $1\le \ell\le n$. These are elements of $\R\cup\{\infty\}$  satisfying the recursion
\begin{eqnarray}\label{eq:recrr}
\hat r_{\ell+1}&=&\left(-\frac1{r_\ell}\cdot \frac{s_\ell}{p_\ell} +\frac{\Lambda}{p_\ell}\right)\left(1+\frac{X_\ell}{p_\ell}\right)^{-1},\qquad 0\le \ell\le n-1\\
 r_{\ell+1}&=&\left(-\frac1{\hat r_{\ell+1}}\cdot \frac{p_{\ell+1}}{s_\ell} +\frac{\Lambda}{s_\ell}\right)\left(1+\frac{Y_\ell}{s_\ell}\right)^{-1},\qquad 0\le \ell\le n-2,\label{eq:recr}
\end{eqnarray}
with initial condition $r_0=\infty$.
We can set $Y_n=0$ and define $r_n$ via (\ref{eq:recr}) with $\ell=n-1$, then $\Lambda$ is an eigenvalue if and only if $r_n=0$.

\subsection{The hyperbolic point of view}

We use the hyperbolic geometric approach of \cite{carousel} to study the evolution of $r$ and $\hat r$. We will view $\R\cup \{\infty\} $ as the boundary of the hyperbolic plane $\HH=\{\Im z>0: z\in \CC\}$ in the Poincar\'e half-plane model. We denote the  group of  linear fractional transformations preserving $\HH\,$ by $\textup{PSL}(2, \R)$.  The recursions for both $r$ and $\hat r$ evolve by elements of this group  of the form $x\mapsto b-a/x$ with $a>0$. 

The Poincar\'e half-plane model
is equivalent to the Poincar\'e disk model $\UU=\{|z|< 1\}$ via the conformal bijection $\mathbf{U}(z)= \frac{i-z}{z-i}$ which is also a bijection between the boundaries $\partial\HH=\R\cup\{\infty\}$ and $\partial\UU=\{|z|=1, z\in \CC\,\}$. Thus elements of $\textup{PSL}(2, \R)$ also act naturally on the unit circle $\partial \UU$. 
By lifting these maps to $\R$, the universal cover of $\partial \UU$, each element $\mathbf{T}$ in $\textup{PSL}(2, \R)$ becomes an $\R\to \R$ function. The lifted versions are uniquely determined up to shifts by $2\pi$ and will also form a group which we denote by $\textup{UPSL}(2, \R)$. For any $\mathbf{T}\in \textup{UPSL}(2, \R)$ we can look at $\mathbf{T}$ as a function acting on $\partial \HH\,$, $\partial \UU$ or $\R$. We will denote these actions by:
\[
\partial \HH\to \partial \HH\,: z\mapsto z\ldot \mathbf{T}, \quad \partial \UU\to \partial \UU\,: z\mapsto z\lcirc \mathbf{T}, \quad \partial \R\to \partial \R\,: z\mapsto z\lstar \mathbf{T}.
\]
For every $\mathbf{T}\in \textup{UPSL}(2, \R)$ the function $x\mapsto f(x)=x\lstar T$ is monotone, analytic and quasiperiodic modulo $2\pi$: $f(x+2\pi)=f(x)+2\pi$. It is clear from the definitions  that $e^{i x}\lcirc T=e^{i f(x)}$ and $(2\tan(x))\ldot \mathbf{T}=2 \tan f(x)$.

Now we will introduce a couple of simple elements of $\textup{UPSL}(2, \R)$. For a given $\alpha\in \R$ we will denote by $\mathbf{Q}(\alpha)$ the rotation by $\alpha$ in $\UU$ about 0. More precisely,
$
\vfi\lstar \mathbf{Q}(\alpha)=\vfi+\alpha
$. For $a>0, b\in \R$ we denote by $\mbf{A}(a,b)$ the affine map $z\to a(z+b)$ in $\HH$\,. This is an element of $\textup{PSL}(2, \R)$ which fixes $\infty$ in $\HH$\, and $-1$ in $\partial \UU$. We specify its lifted version in $\textup{UPSL}(2, \R)$ by making it fix $\pi$, this will uniquely determines it as a $\R\to \R$ function.

Given $\mathbf{T}\in \textup{UPSL}(2, \R)$, $x, y\in \R$ we define the angular shift
\[
\textup{ash}(\mathbf{T}, x, y)=(y\lstar \mathbf{T}-x\lstar \mathbf{T})-(y-x)
\]
which gives the change in the signed distance of $x, y$ under $\mathbf{T}$. This only depends on $v=e^{ix}$, $w=e^{iy}$ and the effect of $\mathbf{T}$ on $\partial U$, so we can also view $\ash(\mathbf{T},\cdot,\cdot)$ as a function on $\partial U\times \partial U$ and the following identity holds:
\[
\ash(\mathbf{T}, v, w)=\arg_{[0,2\pi)}(w\lcirc \mathbf{T}/v\lcirc \mathbf{T})-\arg_{[0,2\pi)}(w/v).
\]
The following lemma appeared as Lemma 16 in \cite{carousel}, it provides a useful estimate for the angular shift.
\begin{lemma}\label{lem:ash}
Suppose  that for a $\mbf{T}\in \textup{UPSL}(2, \R)$ we have $(i+z).\mbf{T}=i$ with
$|z|\le \frac13$. Then
 \begin{equation} \label{eq:ash}
 \begin{array}{rcl}
 \ash(\mbf{T},v,w)&=&\Re\left[(\bar w -\bar v)\left(-z-\frac{i(2+\bar v + \bar w)}{4}\,z^2\right)
 \right]+\eps_3\\[5pt]
  &=&-\Re\left[(\bar w -\bar v) z\right]+\eps_2=\eps_1,
  \end{array}
  \end{equation}
where for $d=1,2,3$ and an absolute constant $c$ we have
\begin{eqnarray}
|\eps_d| \le c |w-v| |z|^d\le 2c |z|^d.\label{eq:asherr}
  \end{eqnarray}
If $v=-1$ then the previous bounds hold even in the case
$|z|>\frac13$.
\end{lemma}

\subsection{Regularized phase functions}

Because of the scaling in (\ref{eq:limit1}) we will set
\[\Lambda
=\mu_n +\frac{\lambda}{4n_0^{{1/2}}}.\]
We introduce the following operators
\begin{eqnarray*}
\mbf{J}_\ell=\mbf{Q}(\pi) \mbf{A}(s_\ell/p_\ell, \mu_n/s_\ell),& \quad &\mbf{M}_\ell=\mbf{A}( (1+X_\ell/p_\ell)^{-1},\lambda/(4 n_0^{1/2} p_\ell))\mbf{A}(\frac{p_{\ell}}{p_{\ell+1}},0),\\\hat {\mbf{J}}_\ell=\mbf{Q}(\pi) \mbf{A}(p_{\ell}/s_\ell, \mu_n/p_{\ell}),
& \quad &\hat {\mbf{M}}_\ell=\mbf{A}( (1+Y_\ell/s_\ell)^{-1},\lambda/(4 n_0^{1/2}s_{\ell})).
\end{eqnarray*}
Then (\ref{eq:recrr}) and (\ref{eq:recr}) can be rewritten as 
\[
r_{\ell+1}=r_\ell\ldot {\mbf{J}}_\ell {\mbf{M}}_\ell \hat{\mbf{J}}_\ell \hat {\mbf{M}}_\ell, \qquad r_0=\infty.
\]
(We suppressed the $\lambda$ dependence in $r$ and the operators $\mbf{M}, \hat {\mbf{M}}$.) Lifting these recursions from $\partial \HH$ to $\R$ we get the evolution of the corresponding phase angle which we denote by $\phi_{\ell}=\phi_{\ell,\lambda}$.
\begin{equation}
\phi_{\ell+1}=\phi_\ell\lstar {\mbf{J}}_\ell {\mbf{M}}_\ell \hat{\mbf{J}}_\ell \hat {\mbf{M}}_\ell, \qquad \phi_0=-\pi.
\end{equation}
Solving the recursion from the other end, with end condition $0$ we get the target phase function $\phi^{\odot}_{\ell,\lambda}$:
\begin{equation}\label{eq:target1}
\phi^{\odot}_{\ell}=\phi^{\odot}_{\ell+1}\lstar  \hat {\mbf{M}}_\ell^{-1} \hat{\mbf{J}}_\ell^{-1}  {\mbf{M}}_\ell^{-1}  {\mbf{J}}_\ell^{-1}, \qquad \phi^{\odot}_n=0.
\end{equation}
It is clear that $\phi_{\ell, \lambda}$ and $\phi^{\odot}_{\ell, \lambda}$ are independent for a fixed $\ell$ (as functions in $\lambda$), they are monotone and analytic in $\lambda$ and we can count eigenvalues using the formula (\ref{evcount}).

Note that ${\mbf{J}}_\ell$ and $\hat{\mbf{J}}_\ell$ do not depend on $\lambda$ and they are not infinitesimal. The main part of the evolution is ${\mbf{J}}_\ell \hat {\mbf{J}}_{\ell}$. This is a rotation in the hyperbolic plane if it only has one fixed point in $\HH$. The fixed point equation $\rho_\ell=\rho_\ell\ldot {\mbf{J}}_\ell \hat {\mbf{J}}_{\ell}$ can be rewritten as 
\[
\rho_\ell=\frac{\rho_\ell s_\ell (s_\ell-\mu_n )+p_\ell \mu_n }{s_\ell (p_\ell-\rho_\ell s_\ell)}.
\]
This can be solved explicitly, and  one gets the 
 following unique solution in the upper half plane if  $\ell< n_0$:
\begin{equation}\label{def:rho}
\rho_\ell=
\frac{\mu_n^2-m+n}{2\mu_n s_\ell}+i \sqrt{1-\frac{(\mu_n^2-m+n)^2}{4\mu_n^2s_\ell^2}}.
\end{equation}
(One also needs to use the identity $p_\ell^2-s_\ell^2=m-n$.)
This shows that if $\ell< n_0$ then ${\mbf{J}}_\ell\hat{\mbf{J}}_\ell$ is a rotation in the hyperbolic plane. We can move the center of rotation to 0 in $\UU$ by conjugating it with an appropriate affine transformation:
\[
\mbf{J}_\ell\hat {\mbf{J}}_\ell=\mbf{Q}(-2 \arg(\rho_\ell \hat \rho_\ell))^{\mbf{T}_\ell^{-1}}.
\]
Here ${\mbf{T}}_\ell={\mbf{A}}(\Im(\rho_\ell)^{-1}, -\Re \rho_\ell)$, $\mbf{X}^\mbf{Y}=\mbf{Y}^{-1}\mbf{X}\mbf{Y}$ and
\begin{equation}\label{def:rrho}
\hat \rho_\ell=\frac{\mu_n^2+m-n}{2\mu_n p_\ell}+i \sqrt{1-\frac{(\mu_n^2+m-n)^2}{4\mu_n^2p_\ell^2}}.
\end{equation}
In order to regularize the evolution of the phase function we introduce 
\[
\varphi_{\ell, \lambda}:=\phi_{\ell,\lambda} \lstar {\mbf{T}}_{\ell} {\mbf{Q}}_{\ell-1}, \qquad 0 \le \ell<n_0
\]
where ${\mbf{Q}}_\ell={\mbf{Q}}(2 \arg(\rho_0\hat \rho_0))\dots {\mbf{Q}}(2 \arg(\rho_\ell\hat \rho_\ell))$ and  ${\mbf{Q}}_{-1}$ is the identity. It is easy to check that the initial condition remains $\vfi_{0,\lambda}=\pi$. Then
\begin{eqnarray*}
\vfi_{\ell+1}&=&\vfi_\ell\lstar {\mbf{Q}}_{\ell-1}^{-1} {\mbf{T}}_{\ell}^{-1}{\mbf{J}}_\ell {\mbf{M}}_\ell \hat{\mbf{J}}_\ell \hat {\mbf{M}}_\ell {\mbf{T}}_{\ell+1} {\mbf{Q}}_{\ell}\\
&=&\vfi_\ell\lstar {\mbf{Q}}_{\ell-1}^{-1} {\mbf{T}}_{\ell}^{-1}{\mbf{Q}}(-2 \arg(\rho_\ell))^{{\mbf{T}}_\ell^{-1}} {\mbf{M}}_\ell^{\hat{\mbf{J}}_\ell} \hat {\mbf{M}}_\ell {\mbf{T}}_\ell {\mbf{T}}_\ell^{-1}{\mbf{T}}_{\ell+1} {\mbf{Q}}_{\ell}\\
&=&\vfi_\ell\lstar \left(\left({\mbf{M}}_\ell^{\hat{\mbf{J}}_\ell}\right)^{{\mbf{T}}_\ell}\right)^{{\mbf{Q}}_\ell} \left((\hat {\mbf{M}}_\ell)^{{\mbf{T}}_\ell}\right)^{{\mbf{Q}}_\ell} \left({\mbf{T}}_\ell^{-1}{\mbf{T}}_{\ell+1}\right)^{ {\mbf{Q}}_{\ell}}
\end{eqnarray*}
Note that the evolution operator is now infinitesimal: ${\mbf{M}}_\ell, \hat {\mbf{M}}_\ell$ and ${\mbf{T}}_\ell^{-1}{\mbf{T}}_{\ell+1}$ are all asymptotically small, and the various conjugations will not change this.

We can also introduce the corresponding target phase function
\begin{equation}\label{eq:target2}
\vfi^{\odot}_{\ell,\lambda}:=\phi^{\odot}_{\ell,\lambda} \lstar {\mbf{T}}_{\ell} {\mbf{Q}}_{\ell-1}, \qquad 0\le \ell<n_0.
\end{equation}
The new, regularized phase functions $\vfi_{\ell, \lambda}$ and $\vfi^{\odot}_{\ell, \lambda}$ 
have the same properties as $\phi, \phi^{\odot}$, i.e.: they
are independent for a fixed $\ell$ (as functions in $\lambda$), they are monotone and analytic  in $\lambda$ and we can count eigenvalues using the formula (\ref{evcount1}).

We will further simplify the evolution using the following identities:
\[
-\frac{a}{r}+b=\left(\frac{b^2+1}{a}r-b\right){\mbf{Q}}\left(\arg\left(\frac{b-i}{b+i}\right)\right), \qquad r\ldot \hat{\mbf{J}}_\ell {\mbf{T}}_\ell=-\frac1{r} \frac{p_\ell}{s_\ell  \Im \rho_\ell}+\frac{\mu_n}{s_\ell \Im \rho_\ell}-\frac{\Re \rho_\ell}{\Im \rho_\ell}.
\]
From this we get
\[
\hat{\mbf{J}}_\ell {\mbf{T}}_\ell=\hat{\mbf{T}}_\ell  {\mbf{Q}}_\ell(-2 \arg(\hat \rho_\ell))
\]
where
\begin{eqnarray*}
r\ldot \hat{\mbf{T}}_\ell&=&\left(\frac{s_\ell}{\Im \rho_\ell p_\ell}-2 \frac{\Re \rho_\ell}{\Im \rho_\ell} \mu_n+\frac{\mu_n^2}{\Im \rho_\ell p_\ell s_\ell}\right) r-\frac{\mu_n}{s_\ell \Im \rho_\ell}+\frac{\Re \rho_\ell}{\Im \rho_\ell}\\
&=&\frac{1}{\Im \hat \rho_\ell} r-\frac{\Re \hat \rho_\ell}{\Im \hat \rho_\ell}.
\end{eqnarray*}
This allows us to write
\begin{equation}
\left(\left({\mbf{M}}_\ell^{\hat{\mbf{J}}_\ell}\right)^{{\mbf{T}}_\ell}\right)^{{\mbf{Q}}_\ell}=({\mbf{M}}_\ell^{\hat{\mbf{T}}_\ell})^{ {\mbf{Q}}(-2 \arg(\hat \rho_{\ell})) {\mbf{Q}}_\ell}=({\mbf{M}}_\ell^{\hat{\mbf{T}}_\ell})^{ \hat  {\mbf{Q}}_\ell}.
\end{equation}
where
\[
\hat {\mbf{Q}}_\ell={\mbf{Q}}_\ell {\mbf{Q}}(-2 \arg(\hat \rho_\ell)).
\]
Thus
\begin{eqnarray*}
\vfi_{\ell+1}
&=&\vfi_\ell\lstar \left({\mbf{M}}_\ell^{\hat{\mbf{T}}_\ell}\right)^{ \hat {\mbf{Q}}_\ell} \left(\hat {\mbf{M}}_\ell^{{\mbf{T}}_\ell}\right)^{{\mbf{Q}}_\ell} \left({\mbf{T}}_\ell^{-1}{\mbf{T}}_{\ell+1}\right)^{ {\mbf{Q}}_{\ell}}.
\end{eqnarray*}
We will introduce the following operators to break up the evolution into smaller pieces:
\begin{eqnarray*}
&&{\mbf{L}}_{\ell, \lambda}={\mbf{A}}(1,\lambda/(4 n_0^{1/2} p_{\ell})),\qquad \hat {\mbf{L}}_{\ell, \lambda}={\mbf{A}}(1,\lambda/(4 n_0^{1/2} s_{\ell})),\\
&&{\mbf{S}}_{\ell,\lambda}={\mbf{L}}_{\ell,\lambda}^{\hat{\mbf{T}}_\ell} \left( {\hat{\mbf{T}}_\ell}^{-1} {\mbf{A}}(\frac{p_{\ell}}{p_{\ell+1}} (1+X_\ell/p_\ell)^{-1},0)\,  \hat{\mbf{T}}_\ell \right),\\
&& \hat {\mbf{S}}_{\ell,\lambda}=\hat {\mbf{L}}_{\ell,\lambda}^{ {\mbf{T}}_\ell} \left( { {\mbf{T}}_\ell}^{-1} {\mbf{A}}( (1+Y_\ell/s_\ell)^{-1},0)\,   {\mbf{T}}_{\ell+1} \right).
\end{eqnarray*}
Then
\begin{eqnarray*}
\vfi_{\ell+1}
&=&\vfi_\ell\lstar \left({\mbf{L}}_\ell^{\hat{\mbf{T}}_\ell}\right)^{ \hat {\mbf{Q}}_\ell}\left({\mbf{S}}_{\ell,0}\right)^{ \hat {\mbf{Q}}_\ell} \left(\hat {\mbf{L}}_\ell^{{\mbf{T}}_\ell}\right)^{{\mbf{Q}}_\ell} \left(\hat {\mbf{S}}_{\ell,0}\right)^{ {\mbf{Q}}_{\ell}}=\vfi_\ell\lstar \left({\mbf{S}}_{\ell,\lambda}\right)^{ \hat {\mbf{Q}}_\ell} \left(\hat {\mbf{S}}_{\ell,\lambda}\right)^{ {\mbf{Q}}_{\ell}}.
\end{eqnarray*}

\section{SDE limit for the phase function} \label{s:sde}

Let $\FF_\ell$ denote the $\sigma$-field generated by $\vfi_{j,\lambda}, j\le \ell$. Then $\vfi_{\ell, \lambda}$ is a Markov chain in $\ell$ with respect to $\FF_{\ell}$. We will show that this Markov chain converges to a diffusion limit after proper normalization. In order to do this we will estimate $\ev\left[ \vfi_{\ell+1,\lambda}-\vfi_{\ell,\lambda}|\FF_\ell\right]$ and $\ev\left[(\vfi_{\ell+1,\lambda}-\vfi_{\ell,\lambda})(\vfi_{\ell+1,\lambda'}-\vfi_{\ell,\lambda'})| \FF_\ell \right]$ using the angular shift lemma, Lemma \ref{lem:ash}.

\subsection{Single step estimates}
Throughout the rest of the proof we will use the notation $k=n_0-\ell$. We will need to rescale the discrete time $n_0$ in order to get a limit, we will use  $t=\ell/n_0$ and also introduce $\hat s(t)=\sqrt{1-t}$.
We start with the identity
\begin{eqnarray*}
p_\ell \Im \hat \rho_\ell=s_\ell \Im \rho_\ell&=&\sqrt{s_\ell^2-
\frac{(\mu_n^2-m+n)^2}{4\mu_n^2}}=\sqrt{n-\frac{(\mu_n^2-m+n)^2}{4\mu_n^2}-\ell-\frac12}\\
&=&\sqrt{n_0-\ell}=\sqrt{k}=\sqrt{n_0} \hat s(t).
\end{eqnarray*}
Note that this means that
\begin{eqnarray}\label{eq:rhodef}
\rho_\ell&=&\pm \sqrt{\frac{n-n_0-1/2}{n-\ell-1/2}}+i \sqrt{\frac{n_0-\ell}{n-\ell-1/2}}
=\pm \sqrt{\frac{n_1}{n_1+k}}+i \sqrt{\frac{k}{n_1+k}},
\\
\hat \rho_\ell&=& \sqrt{\frac{m-n_0-1/2}{m-\ell-1/2}}+i \sqrt{\frac{n_0-\ell}{m-\ell-1/2}}=\sqrt{\frac{m_1}{m_1+k}}+i \sqrt{\frac{k}{m_1+k}}\label{eq:hrhodef}
\end{eqnarray}
where the sign in $\Re \rho_\ell$ is positive if $\mu_n>\sqrt{m-n}$ and negative otherwise.

For the angular shift estimates we need to consider
\begin{eqnarray}\nonumber
Z_{\ell,\lambda}&=&i.{\mbf{S}}_{\ell,\lambda}^{-1}-i=\frac{\hat \rho_\ell X_\ell}{\sqrt{n_0}\hat s(t)}\cdot \frac{p_{\ell+1}}{p_\ell}+\left(-\frac{\lambda}{4n_0\hat s(t)}
+\frac{\hat \rho_\ell(p_{\ell+1}-p_{\ell})}{p_{\ell} \Im\hat\rho_\ell}
\right)=V_\ell+v_\ell,\\
\hat Z_{\ell,\lambda}&=&i. \hat {\mbf{S}}_{\ell,\lambda}^{-1} -i=\frac{\rho_\ell Y_\ell}{\sqrt{n_0}\hat s(t)}+\left(-\frac{\lambda}{4n_0\hat s(t)}+\frac{\rho_{\ell+1}-\rho_{\ell}}{\Im \rho_\ell}\right)=\hat V_\ell+\hat v_\ell.\label{eq:ZZ}
\end{eqnarray}
We have the following estimates for the deterministic part (by Taylor expansion):
\begin{eqnarray*}
v_{\ell,\lambda}&=&\frac{v_{\lambda}(t)}{n_0}+O(k^{-2}), \quad v_{\lambda}(t)=-\frac{\lambda}{4  \hat s(t)}-\frac{\hat \rho(t)}{2 p(t) \hat s(t)}, \quad |v_{\ell,\lambda}|\le  \frac{c}{k},\\
\hat v_{\ell,\lambda}&=&\frac{\hat v_{\lambda}(t)}{n_0}+O(k^{-2}), \quad \hat v_{\lambda}(t)=-\frac{\lambda}{4  \hat s(t)}+\frac{\frac{d}{dt}\rho(t)}{\Im \rho(t)}, \quad |\hat v_{\ell,\lambda}|\le \frac{c}{k},
\end{eqnarray*}
where $p(t)=p^{(n)}(t)=\sqrt{m/n_0-t}$ and
$\rho(t)=\rho^{(n)}(t)$,  $\hat \rho(t)=\hat \rho^n(t)$ are defined by equations (\ref{eq:rhodef}) and (\ref{eq:hrhodef}) with $\ell=n_0 t$.
For the random terms from (\ref{eq:moments}) we get
\begin{eqnarray*}
\ev V_\ell&=&\ev\hat V_\ell=\mathcal{O}(k^{-1/2}(n-\ell)^{-3/2}), \\
\ev V_\ell^{2}&=&\frac1{n_0}q^{(1)}(t)+\mathcal{O}(k^{-1}(n-\ell)^{-1}),   \qquad  \ev \hat V_\ell^{2}=\frac1{n_0}q^{(2)}(t)+\mathcal{O}(k^{-1}(n-\ell)^{-1}),\\
\ev|V_\ell^{2}|&=&\ev|\hat V_\ell^{2}|=\frac1{n_0}q^{(3)}(t)+\mathcal{O}(k^{-1}(n-\ell)^{-1}),\quad \ev|V_\ell^{d}|, \ev|\hat V_\ell^{2}|=\mathcal{O}(k^{-d/2}), \, d=3,4,
\end{eqnarray*}
where the constants in the error term only depend on $\beta$ and 
\begin{equation}
\hskip-30ptq^{(1)}(t)=\frac{2\hat{\rho}(t)^{2}}{\beta \hat s(t)^{2} }, \qquad q^{(2)}(t)=\frac{2\rho(t)^{2}}{\beta  \hat s(t)^{2} },\qquad q^{(3)}(t)=\frac{2}{\beta  \hat s(t)^{2} }.
\end{equation}
We introduce the notations
\[
\vfi_{\ell+\nicefrac{1}{2}\, ,\lambda}=\vfi_\ell\lstar \left({\mbf{S}}_{\ell,\lambda}\right)^{ \hat {\mbf{Q}}_\ell},\qquad \FF_{\ell+\half}=\sigma(\FF_\ell \cup \{\vfi_{\ell+\half,\lambda}\})
\]
and $\Delta_{\half} f_{x,\lambda}=f_{x+\nicefrac{1}{2}\, ,\lambda}-f_{x,\lambda}$, $\Delta f_{x,\lambda}=f_{x+1 ,\lambda}-f_{x,\lambda}$. 
We also set for $\ell \in \Z_+$
\[
\eta_\ell=\rho_0^2 \hat\rho_0^2\rho_1^2\hat\rho_1^2\dots \rho_\ell^2 \hat\rho_\ell^2.
\]
\begin{remark}
We would like to note that the `half-step' evolution rules $\vfi_{\ell,\lambda}\to \vfi_{\ell+\half,\lambda}$, $\vfi_{\ell+\half,\lambda}\to \vfi_{\ell+1,\lambda}$ are \emph{very similar} to the one-step evolution of the phase function $\vfi$ in \cite{carousel}. Besides the fact that the $\ell\to \ell+\half$ and $\ell+\half\to \ell+1$ steps are quite different from each other the other big difference between our case and \cite{carousel} is that here the oscillating terms $\mbf{Q}_\ell, \hat{\mbf{Q}}_\ell$ are  more complicated.

\end{remark}

The following proposition is the analogue of Proposition 22 in \cite{carousel}.
\begin{proposition}\label{prop:onestepfi}
For $\ell\le n_0$ we have
\begin{eqnarray*}
&&\ev\left[\Delta_{\half} \vfi_{\ell,\lambda}\big| \vfi_{\ell, \lambda}=x \right]=
\frac1{n_0}b^{(1)}_\lambda(t)+\frac1{n_0} osc^{(1)}+\mathcal{O}(k^{-3/2})=\mathcal{O}(k^{-1})
\\
&&\ev\left[\Delta_{\half}\vfi_{\ell,\lambda} \Delta_{\half} \vfi_{\ell,\lambda'}\big| \vfi_{\ell, \lambda}=x, \vfi_{\ell, \lambda'}=y \right]=\frac1{n_0} a^{(1)}(t,x,y)+\frac1{n_0} osc^{(2)}+\mathcal{O}(k^{-3/2})\\
&&\ev\left[\Delta_{\half} \vfi_{\ell+\half,\lambda}\big| \vfi_{\ell+\half, \lambda}=x \right]=
\frac1{n_0}b^{(2)}_\lambda(t)+\frac1{n_0} osc^{(3)}+\mathcal{O}(k^{-3/2})=\mathcal{O}(k^{-1})
\\
&&\ev\left[\Delta_{\half}\vfi_{\ell+\half,\lambda} \Delta_{\half} \vfi_{\ell+\half,\lambda'}\big| \vfi_{\ell+\half, \lambda}=x, \vfi_{\ell+\half, \lambda'}=y \right]=\frac1{n_0} a^{(2)}(t,x,y)+\frac1{n_0} osc^{(4)}+\mathcal{O}(k^{-3/2}),\\
&&\ev \left[|\Delta_{\half} \vfi_{\ell,\lambda}|^d\big|
\vfi_{\ell, \lambda}=x
\right]
, \ev \left[ |\Delta_{\half} \vfi_{\ell+\half,\lambda}|^d\big|
\vfi_{\ell, \lambda}=x
\right]=\mathcal{O}(k^{-d/2}), \qquad d=2,3
\end{eqnarray*}
where $t=\ell/n_0$, 
\begin{eqnarray*}
&b_\lambda^{(1)}=\frac{\lambda}{4 \hat s}+\frac{\Re \hat \rho}{2p \hat s}+\frac{\Im \hat \rho^2}{2 \beta \hat s^2},
\qquad &b_\lambda^{(2)}=\frac{\lambda}{4 \hat s}-\frac{\Re{\frac{d}{dt} \rho}}{\Im \rho}+\frac{\Im  \rho^2}{2 \beta \hat s^2},\\
&a^{(1)}=\frac{1}{\beta \hat s^2} \Re\left[e^{i (x-y)}\right]+\frac{1+\Re \hat \rho^2}{\beta \hat s^2},\qquad &a^{(2)}=\frac{1}{\beta \hat s^2} \Re\left[e^{i (x-y)}\right]+\frac{1+\Re   \rho^2}{\beta \hat s^2}.
\end{eqnarray*}
The oscillatory terms are
\begin{eqnarray*}
osc^{(1)}&=&\Re((- v_\ell- i q^{(1)}/2)e^{-ix} \hat \rho_\ell^{-2} \eta_\ell)+\Re(i e^{-2ix} \hat \rho_\ell^{-4}\eta_\ell^2 q^{(1)})/4, \\
osc^{(2)}&=&q^{(3)} \Re(e^{-ix} \hat \rho_\ell^{-2}\eta_\ell+e^{-i y} \hat \rho_\ell^{-2}\eta_\ell)/2+\Re(q^{(1)}(e^{-ix} \hat \rho_\ell^{-2}\eta_\ell+e^{-i y} \hat \rho_\ell^{-2}\eta_\ell+e^{-i(x+y)}\hat \rho_\ell^{-4}\eta_\ell^2))/2,\\
osc^{(3)}&=&\Re((- \hat v_{\ell}- i q^{(2)}/2)e^{-ix} \eta_{\ell})+\Re(i e^{-2ix} \eta_{\ell}^2 q^{(2)})/4, \\
osc^{(4)}&=&q^{(3)} \Re(e^{-ix} \eta_{\ell}+e^{-i y} \eta_{\ell})/2+\Re(q^{(2)}(e^{-ix} \eta_{\ell}+e^{-i y} \eta_{\ell}+e^{-i(x+y)}\eta_{\ell}^2))/2.
\end{eqnarray*}
\end{proposition}
\begin{proof}
We start with the identity
\[
\vfi_{\ell+\half,\lambda}-\vfi_{\ell,\lambda}=\vfi_{\ell+1,\lambda}\lstar \hat{\mbf{Q}}_{\ell}^{-1}-\vfi_{\ell,\lambda}\lstar \hat{\mbf{Q}}_{\ell}^{-1}=\vfi_{\ell,\lambda}\lstar \hat{\mbf{Q}}_{\ell}^{-1} \mbf{S}_{\ell, \lambda}
-\vfi_{\ell,\lambda}\lstar \hat{\mbf{Q}}_{\ell}^{-1}=\ash( \mbf{S}_{\ell, \lambda}, e^{i \vfi_{\ell,\lambda}} \bar \eta_\ell \hat \rho_\ell^{-2},-1).
\]
Here we used the definition of the angular shift with the fact that $\mbf{S}_{\ell, \lambda}$ (and any affine transformation) will preserve $\infty\in \HH$ which corresponds to $-1$ in $\UU$. A similar identity can be proved for $\Delta_{\half}\vfi_{\ell+\half,\lambda}$.

The proof now follows exactly the same as in \cite{carousel}, it is a straightforward application of Lemma \ref{lem:ash} using the estimates on $v_{\ell, \lambda}, \hat v_{\ell, \lambda}$, $V_{\ell}, \hat V_{\ell}$.
\end{proof}

\subsection{The continuum limit}

In this section we will prove that $\vfi^{(n)}(t,\lambda)=\vfi_{\lfloor t n_0 \rfloor, \lambda}$ converges to the solution of a one-parameter family of stochastic differential equations on $t\in[0,1)$. The main tool is the following proposition, proved in \cite{carousel} (based on \cite{SV} and \cite{EthierKurtz}).

\begin{proposition}\label{p_turboEK}
Fix $T>0$, and for each $n\ge 1$ consider a Markov chain
$
X^n_\ell\in \R^d$ with $\ell =1\ldots  \lfloor nT
\rfloor$.
Let $Y^n_\ell(x)$ be distributed as the increment
$X^n_{\ell+1}-x$ given $X^n_\ell=x$. We define
$$
b^n(t,x)= n E[ Y_{\lfloor nt\rfloor}^n(x)],\qquad
a^n(t,x)=nE[ Y_{\lfloor nt\rfloor}^n(x) Y_{\lfloor
nt\rfloor}^n(x)^\textup{T}].
$$
Suppose that as $n\to \infty $ we have
\begin{eqnarray}
|a^n(t,x)-a^n(t,y)|+|b^n(t,x)-b^n(t,y)|&\le&
c|x-y|+o(1)\label{eq:lip}\\
 \sup_{x,\ell}  E[|Y^n_\ell(x)|^3] &\le& cn^{\nicefrac{-3}{2}} \label{eq:3m},
\end{eqnarray}
and that there are functions $a,b$ from $\R\times[0,T]$ to
$\R^{d^2}, \R^d$ respectively with bounded first and second
derivatives so that
\begin{eqnarray}
\sup_{x\in \R^{d^2},t} \Big|\int_0^t a^n(s,x)\,ds-\int_0^t a(s,x)\,ds
\Big|+\sup_{x\in \R^{d},t} \Big|\int_0^t b^n(s,x)\,ds-\int_0^t b(s,x)\,ds
\Big| &\to& 0. \label{eq:avr}
\end{eqnarray}
Assume also that the initial conditions converge weakly, 
$
X_0^n\cd X_0.
$

Then $(X^n_{\lfloor n t\rfloor}, 0 \le t \le T)$ converges in law
to the unique solution of the SDE
\begin{equation*}
dX = b \,dt + \sigma\, dB, \qquad X(0)=X_0, \quad t\in[0,T],
\end{equation*}
where $B$ is a $d$-dimensional standard Brownian motion and $\sigma: \R^d\times [0,T]$ is a square root of the matrix valued function $a$, i.e.~$a(t,x)=\sigma(t,x)\,\sigma(t,x)^T$.
\end{proposition}
We will apply this proposition to $\varphi_{\ell,\lambda}$ with $\ell \le n_0(1-\eps)$ and $\ell \in \Z/2$, so the single steps of the proposition correspond to half steps in our setup.

The following lemma shows that the oscillatory terms in the estimates of Proposition \ref{prop:onestepfi} average out in the `long run'. Its proof relies on Proposition \ref{prop:onestepfi} and Lemma \ref{l_sum} of the Appendix.
\begin{lemma}\label{lem:bsum}
Let  $|\lambda|, |\lambda'|\le \bar \lambda$ and $\eps>0$. Then for any $\ell_1\le n_0
(1-\eps)$, $\ell_1 \in \Z$
\begin{align}\label{eq:bsum}
&\frac1{n_0}\sum_{0\le \ell< \ell_1}^{\sim}
\ev\left[\Delta_{\half} \vfi_{\ell,\gl}\,|\,\varphi_{\ell,\gl}=x
\right]
=\frac{1}{n_0}\sum_{\ell=0}^{ \ell_1-1 } b_\gl(t) +\O(n_0^{-\nicefrac{1}{2}}+n_1^{1/2} n_0^{-3/2})\\
&\frac{1}{n_0}\sum_{0\le \ell< \ell_1}^{\sim}
\ev\left[\Delta_{\half}\vfi_{\ell,\gl}\Delta_{\half}
\vfi_{\ell,\gl'}\,|\,\varphi_{\ell,\gl}=x ,\,
\varphi_{\ell,\gl'}=y \right]
=\frac{1}{n_0}\sum_{\ell=0}^{\ell_1-1 }
a(t,x,y)+
\O( n_0^{-\nicefrac{1}{2}}+n_1^{1/2} n_0^{-3/2}) \nonumber
\end{align}
where $t=\ell/n_0$, the functions $b_\gl, a$ are defined as
\begin{eqnarray}\label{eq:coeffs}
b_\lambda&=&\frac{\lambda}{2 \hat s}+\frac{\Re \hat \rho}{2p \hat s}+\frac{\Im (\hat \rho^2+\rho^2)}{2 \beta \hat s^2}-\frac{\Re{\frac{d}{dt} \rho}}{\Im \rho}, \quad
a=\frac{2}{\beta \hat s^2} \Re\left[e^{i (x-y)}\right]+\frac{2+\Re (\hat \rho^2+\rho^2)}{\beta \hat s^2},
\end{eqnarray}
and the implicit constants in $\O$ depend
only on  $\eps,\beta, \bar \lambda$. The indices in the summation $\sum\limits^{\sim}$ run through half integers.
\end{lemma}
\begin{proof}[Proof of Lemma \ref{lem:bsum}]
We will only prove the first statement, the second one being similar. Note that $b_\lambda(t)=b^{(1)}_\lambda(t)+b^{(2)}_\lambda(t)$.

Summing the first and third estimates in Proposition \ref{prop:onestepfi} we get (\ref{eq:bsum}) with an error term
\begin{eqnarray}\label{bsum1}
\frac1{n_0}\sum_{\ell=0}^{\ell_1-1}  \Re (e_{1,\ell}\,
\eta_{\ell})+\frac1{n_0}\sum_{\ell=0}^{\ell_1-1}  \Re
(e_{2,\ell} \,\eta_{\ell}^2)+\O(n_0^{\nicefrac{-1}{2}}),
\end{eqnarray}
where the first two terms will be denoted
$\zeta_1,\zeta_2$.
Here
$$e_{1,\ell}=\left((- v_\lambda- i q^{(1)}/2)\hat \rho_\ell^2+(- \hat v_{\lambda}- i q^{(2)}/2)\right)e^{-i  x},\qquad e_{2,\ell}=i (\hat \rho_\ell^{-4} q^{(1)}+q^{(2)}) e^{- 2i
x}/4$$
where for this proof $c$
denotes varying constants depending on $\eps$. Using the
fact that $v_\gl, \hat v_\lambda,  q^{(1)}, q^{(2)}$ and their first derivatives are
continuous on $[0,1-\eps]$
 we get
\begin{equation}\label{eilbounds}
|e_{i,\ell}|<c,\qquad |e_{i,\ell}-e_{i,\ell+1}|<c
n_0^{-1}.\end{equation}
Applying Lemma \ref{l_sum} of the Appendix to the first sum in (\ref{bsum1}):
\[
|\zeta_1|\le \frac1{n_0} |e_{1,\ell_1}| |F^{(1)}_{1,\ell_1}|+\frac1{n_0}\sum_{\ell=1}^{\ell_1-1}|e_{1,\ell}-e_{1,\ell+1}| |F^{(1)}_{1,\ell} |.
\]
Since $\ell_1\le n_0(1-\eps)$ we have $|F^{(1)}_{1,\ell}|\le c(1+n_1^{1/2} k^{-1/2})\le c(n_1^{1/2} n_0^{-1/2}+1)$ and
\[
|\zeta_1|\le c (n_0^{-3/2} n_1^{1/2}+n_0^{-1}).
\]
(Recall that $k=n_0-\ell$.) For the estimate of $\zeta_2$ we
first note that
\begin{equation}
|e_{2,\ell}|=\frac{1}{2\beta} \frac{n_0}{k}|\hat \rho_\ell^{-2}+\rho_\ell^2|=\frac{1}{2\beta} \frac{n_0}{k}|\hat \rho_\ell^{2}\rho_\ell^2+1|.\label{eq:direct}
\end{equation}
We will use Lemma \ref{l_sum} if $|\hat \rho_\ell^{2}\rho_\ell^2+1|$ is `big', and a direct bound with (\ref{eq:direct}) if it is small.
To be more precise: we divide the sum into three pieces, we cut it at indices $\ell_1^*$ and $\ell_2^*$ so that
\begin{equation}\label{case1a}
|\hat \rho_\ell^{2}\rho_\ell^2+1|\le n_0^{-1/2} \quad \textup{if } k\in[k_2^*,k_1^*]\qquad \textup{and }|\hat \rho_\ell^{2}\rho_\ell^2+1|\ge n_0^{-1/2} \quad \textup{otherwise}.
\end{equation}
Note that one or two of the resulting partial sums may be empty.  We can always find such indices because $\arg \hat \rho_\ell^{2}\rho_\ell^{2}$ is monotone if $\mu_n\ge \sqrt{m-n}$ and if  $\mu_n< \sqrt{m-n}$
then $\arg \hat \rho_\ell^{2}\rho_\ell^{2}$ decreases if $k>\sqrt{m_1 n_1}$ then it {increases}. (See the proof of Lemma \ref{l_sum}.)

We denote the three pieces by $\zeta_{2,i}, i=1,2,3$ and bound them separately. Since $k\ge \eps n_0$, Lemma \ref{l_sum} gives
\[
|\zeta_{2,1}|\le c (n_1^{1/2} n_0^{-3/2}+n_0^{-1/2}).
\] 
The term $|\zeta_{2,3}|$ can be bounded exactly the same way, so we only need to deal with $\zeta_{2,2}$.  Here we use (\ref{eq:direct}) to get a  direct estimate:
\begin{eqnarray*}
|\zeta_{2,2}|&\le& \frac1{2\beta}   \sum_{k\in[k_2^*,k_1^*]\cap[\eps n_0,n_0]} \frac1{k } |\hat \rho_\ell^{-2}+\rho_\ell^2|\le c n_0^{-1/2}.
\end{eqnarray*}
Collecting all our estimates the statement follows.
%
%
%
\end{proof}
Now we have the ingredients to prove the continuum limit.
\begin{proposition}\label{prop:phisde}
Suppose that $m/n_0\to \kappa \in[1,\infty]$, 
$n/n_0\to \nu \in [1,\infty]$ and that eventually $\mu_n>\sqrt{m-n}$ or $\mu_n\le \sqrt{m-n}$.
Then the continuous functions $p^{(n)}(t)^{-1}, \rho^{(n)}(t), \hat \rho^{(n)}(t)$  converge to following limits on $[0,1)$:
\[
p^{-1}(t)=(\kappa-t)^{-1/2},\quad \rho(t)=\pm \sqrt{\frac{\nu-1}{\nu-t}}+i \sqrt{\frac{1-t}{\nu-t}},\quad \hat \rho(t)=\sqrt{\frac{\kappa-1}{\kappa-t}}+i \sqrt{\frac{1-t}{\kappa-t}},
\]
where the sign in $\Re \rho$ depends on the (eventual) sign of $\mu_n-\sqrt{m-n}$.
If $\kappa=\infty$ then $p^{-1}(t)=0$, $\hat \rho(t)=1$ and if $\nu=\infty$ then $\rho(t)=\pm 1$.

Let $B$ and $W$ be  independent real and complex standard Brownian motions,
and for each $\lambda\in \R$ consider the strong solution of
\begin{eqnarray}\label{eq:sdelimfi}
  d \varphi_{\lambda}
  &=&
\left[ \frac{\lambda}{2\hat s}-\frac{\Re \rho'}{\Im \rho}
+\frac{\Im (\rho^2+\hat \rho^2)}{2\beta \hat s ^2}+\frac{\Re \hat \rho}{2p \hat s} \right]dt +
\frac{\sqrt2\Re(e^{-i\varphi_\lambda}dW)}{\sqrt{\beta}\,\hat s}
  +\frac{\sqrt{2+\Re (\rho^2+\hat \rho^2)}}{\sqrt{\beta}\,\hat s}dB, \nonumber
  \\[.4em]
  \varphi_{\lambda}(0)&=&\pi.
\end{eqnarray}
Then we have
$$
 \varphi_{\gl, \lfloor n_0 t\rfloor }\cd \varphi_\lambda(t),\qquad \textup{as  $n\to \infty$},
$$
where the convergence is in the sense of finite dimensional
distributions for $\lambda$ and in path-space $D[0,1)$ for
$t$.
\end{proposition}
\begin{proof}
The proof is very similar to the proof of Theorem 25 in \cite{carousel}. One needs to check that for any fixed vector $(\lambda_1,\dots, \lambda_d)$ the Markov chain $(\vfi_{\ell, \lambda_i},1\le i \le d),  \ell\le \lfloor (1-\eps) n_0\rfloor$, $\ell \in \Z/2$ satisfies the conditions of Proposition \ref{p_turboEK} and to identify the variance matrix of the limiting diffusion. 
Note that because our Markov chain lives on the half integers one needs to slightly rephrase the proposition, but this is straightforward.

The Lipshitz condition (\ref{eq:lip}) and the moment condition (\ref{eq:3m}) are easy to check from Proposition \ref{prop:onestepfi}. The averaging condition (\ref{eq:avr}) is satisfied because of Lemma \ref{lem:bsum}, using the fact that because of the conditions of the proposition, the functions $b^\lambda(t), a(t,x,y)$ converge. This proves that the rescaled version of  $(\vfi_{\ell, \lambda_j},1\le j \le d)$ converges in distribution to an SDE in $\R^d$ where the drift term is given by the limit of $(b^\lambda_j, j=1\dots d)$ and the diffusion matrix is given by $a(t,x)_{j,k}=\frac{2}{\beta \hat s^2} \Re\left[e^{i (x_k-x_j)}\right]+\frac{2+\Re (\hat \rho^2+\rho^2)}{\beta \hat s^2}$.

The only step left is to verify that the limiting SDE can be rewritten in the form (\ref{eq:sdelimfi}). This follows easily using the fact that if $Z$ is a complex Gaussian with i.i.d.~standard real and imaginary parts and $\omega_1, \omega_2\in \CC$ then
\[
\ev \Re(\omega_1 Z) \Re(\omega_2 Z)=\Re(\bar \omega_1  \omega_2).\qedhere
\]
\end{proof}

The following corollary describes the scaling limit of the relative phase function $\alpha_{\ell, \lambda}$.
\begin{corollary}\label{cor:sineb}
Let $Z$ be a complex Brownian motion with i.i.d.~standard real and imaginary parts and consider the strong solution $\alpha_{\lambda}(t)$ of the SDE system (\ref{sineb}).
Then $\alpha_{\lfloor n_0 t\rfloor, \lambda}\cd \alpha_{\lambda}(t)$ as $n\to \infty$ where the convergence is in the sense of finite dimensional
distributions for $\lambda$ and in path-space $D[0,1)$ for
$t$.
\end{corollary}
\begin{proof}
We just need to show that for any subsequence of $n$ we can choose a further subsequence so that the convergence holds. By choosing an appropriate subsequence we can assume that $m/n_0, n/n_0$ both converge and that $\mu_n-\sqrt{m-n}$ is always positive or nonnegative. Then the conditions of Proposition \ref{prop:phisde} are satisfied and $\alpha_\lambda=\vfi_\lambda-\vfi_0$ will satisfy the SDE (\ref{sineb}) with a complex Brownian motion $Z_t:=\int_0^t e^{i \vfi_0(t)} dW_t$. From this the statement of the corollary follows.
\end{proof}

\section{Middle stretch}\label{s:middle}

In this section we will study the behavior of $\alpha_{\ell, \lambda}$ and $\vfi_{\ell, \lambda}$ in the interval $[ \lfloor (1-\eps) n_0 \rfloor , n_2]$ with $n_2=\left\lfloor n_0-\mathcal{K}(n_1^{1/3}{\vee} 1)\right\rfloor$. The constant $\mathcal{K}$ will eventually go to $\infty$, so we can assume that $\mathcal{K}>C_0>0$ with $C_0$ large enough.

\subsection{The relative phase function}
The objective of this subsection is to show that the relative phase function $\alpha_{\ell, \lambda}$ does not change much in the middle stretch. 
\begin{proposition}\label{prop:middle}
There exists a constant $c=c(\bar\lambda,\beta)$ so that with $y=n_0^{-1/2}(n_1^{1/6}{\vee} \log n_0)$ we have
\begin{equation}\mathbb{E}\left[|(\alpha_{\ell_2,\lambda}-\alpha_{\ell_1,\lambda})\wedge 1|\mathcal{F}_{\ell_1}\right]\leq c
\left(d(\alpha_{\ell_1,\lambda},2\pi\mathbb{Z})+\sqrt{\epsilon}+y+\mathcal{K}^{-1}\right)\label{eq:alpha}
\end{equation}
for all $\mathcal{K}>0, \epsilon\in(0,1),\lambda\leq |\bar\lambda|, n_0(1-\eps) \leq l_1 \leq l_2 \leq n_2$, $\ell \in \Z$.
\end{proposition}
Because of the moment bounds (\ref{eq:moments}) we may assume that 
\begin{equation}
|X_\ell|, |Y_\ell|\le \frac{1}{10} \sqrt{n_0} \hat s(\ell/n_0), \quad \textup{for $\ell\le n_2$}.\label{eq:cutoff}
\end{equation}
Indeed, the probability that (\ref{eq:cutoff}) does not hold is at most $c(n_0-n_2)^{-1}\le c \mathcal{K}^{-1}$ which can be absorbed in the error term of (\ref{eq:alpha}).

We first provide the one-step estimates for the evolution of the relative phase function.
\begin{proposition}\label{prop:alpha1step}
There exists $c=c(\beta,\bar\lambda)$ so that for every $\ell\leq n_2$ and $|\lambda|<\bar\lambda$ we have the following estimates
\begin{align}\nonumber
&\mathbb{E}\left(\Delta\alpha_{\ell,\lambda}|\FF_\ell\right)=
-\frac{1}{n_0}\Re\left\{\eta_\ell\left(e^{-i\vfi_{\ell,\lambda}}-e^{-i\vfi_{\ell,0}}\right)\right[\hat\rho_\ell^{-2}\left(v_{\lambda}+iq^{(1)}/2\right)+\left(\hat v_{\lambda}+iq^{(2)}/2\right)\left]\right\}\\ \nonumber
&\quad\quad\quad-\frac{1}{n_0}\Re\left\{i\eta_\ell^2/4\left(e^{-2i\vfi_{\ell,\lambda}}-e^{-2i\vfi_{\ell,0}}\right)\right[\hat\rho_\ell^{-4}q^{(1)}+q^{(2)}\left]\right\} +\mathcal{O}(\hat\alpha_{\ell,\lambda}k^{-3/2}+k^{-1/2}n_0^{-1/2})\\
&\hskip 48mm =\mathcal{O}(\hat\alpha_{\ell,\lambda}k^{-1}+k^{-1/2}n_0^{-1/2}) \label{eq:alpha1}\\
&\mathbb{E}\left(\Delta\alpha_{\ell,\lambda}^2|\FF_\ell\right)=\mathcal{O}(\hat\alpha_{\ell,\lambda}k^{-1}+k^{-1}
n_0^{-1})\label{eq:alpha2}\\
&\mathbb{E}\left(|\Delta\alpha_{\ell,\lambda}\Delta\varphi_{\ell,\lambda}|\big|\FF_\ell\right)=
\mathcal{O}(\hat\alpha_{\ell,\lambda}k^{-1})\label{eq:alpha3}
\end{align}
where $\hat\alpha_{\ell,\lambda}$ denotes the distance between $\alpha_{\ell,\lambda}$ and $2\pi.$
\end{proposition}
\begin{proof} We first prove estimates on $\Delta_{\half} \alpha_{\ell, \lambda}$ and $\Delta_{\half} \alpha_{\ell+\half, \lambda}$. In order to do this, we break up the evolution of $\vfi_{\ell, \lambda}$  into even smaller pieces:
\begin{equation}
\vfi_{\ell,\lambda}\stackrel{{\mbf{L}}_{\ell}^{\hat {\mbf{T}}_{\ell}\hat {\mbf{Q}}_{\ell}}}{\longrightarrow}\vfi_{\ell+\qt,\lambda}\stackrel{{\mbf{S}}_{\ell,0}^{\hat {\mbf{Q}}_{\ell}}}{\longrightarrow}\vfi_{\ell+\half,\lambda}\stackrel{\hat{\mbf{L}}_{\ell}^{ {\mbf{T}}_{\ell} {\mbf{Q}}_{\ell}}}{\longrightarrow}\vfi_{\ell+\tqt,\lambda}\stackrel{\hat {\mbf{S}}_{\ell,0}^{ {\mbf{Q}}_{\ell}}}{\longrightarrow}\vfi_{\ell+1,\lambda} \label{eq:fievo}
\end{equation}
where $\vfi_{\ell+\qt,\lambda}$ and $\vfi_{\ell+\tqt,\lambda}$ are defined accordingly. We also define the relative phase functions for the intermediate steps in the natural way.

By choosing $c(\beta,\bar \lambda)$ large enough we can assume $\frac{\bar \lambda}{4 \sqrt{n_0 k}}\le \frac1{10}$ for $\ell \le n_2\le n- \mathcal{K}$. Using this with the cutoff (\ref{eq:cutoff}) the random variables $Z_{\ell, \lambda}, \hat Z_{\ell, \lambda}$ defined in (\ref{eq:ZZ}) are both less than $1/3$ in absolute value. This means that we are allowed to use Lemma \ref{lem:ash} in the general case for each operator appearing in (\ref{eq:fievo}) (i.e. the condition $|z|\le 1/3$ is always satisfied). From this point the proof is similar to the proof of 
 Proposition 29 in \cite{carousel}. We first write
\begin{eqnarray}\nonumber
\Delta_{\half} \alpha_{\ell,\lambda}&=&\ash({\mbf{L}}_{\ell}^{\hat {\mbf{T}}_{\ell}},-1,e^{i \vfi_{\ell,\lambda}} \bar \eta_\ell \hat \rho_{\ell}^{-2})+\ash({\mbf{S}}_{\ell,0}^{\hat {\mbf{Q}}_{\ell}}, e^{i \vfi_{\ell+\qt,\lambda}} \bar \eta_\ell \hat \rho_{\ell}^{-2},e^{i \vfi_{\ell,\lambda}} \bar \eta_\ell \hat \rho_{\ell}^{-2})\\&&\qquad+\ash({\mbf{S}}_{\ell,0}^{\hat {\mbf{Q}}_{\ell}}, e^{i \vfi_{\ell,\lambda}} \bar \eta_\ell \hat \rho_{\ell}^{-2},e^{i \vfi_{\ell,0}} \bar \eta_\ell \hat \rho_{\ell}^{-2})\label{eq:alphaevo}.
\end{eqnarray}
Using Lemma \ref{lem:ash} one can show that the first two terms in (\ref{eq:alphaevo}) are of $\O(n_0^{-1/2} k^{-1/2})$. Using Lemma \ref{lem:ash} again for the third term together with 
\[
|e^{i \varphi_{\ell,\gl}}-e^{i \varphi_{\ell,0}}|=|e^{i \alpha_{\ell,\gl}}-1|\le \hat \alpha_{\ell,\gl}, \qquad |e^{i 2\varphi_{\ell,\gl}}-e^{i 2\varphi_{\ell,0}}|\le 2\hat \alpha_{\ell,\gl}
\]
we get the analogue of (\ref{eq:alpha1}) for $\Delta_{\half} \alpha_{\ell,\lambda}$:
\begin{align}\nonumber
&\mathbb{E}\left(\Delta_{\half}\alpha_{\ell,\lambda}|\FF_\ell\right)=
-\frac{1}{n_0}\Re\left\{\eta_\ell\left(e^{-i\vfi_{\ell,\lambda}}-e^{-i\vfi_{\ell,0}}\right)\right[\hat\rho_\ell^{-2}\left(v_{\lambda}+iq^{(1)}/2\right)\left]\right\}\\ \nonumber
&\quad\quad\quad-\frac{1}{n_0}\Re\left\{i\eta_\ell^2/4\left(e^{-2i\vfi_{\ell,\lambda}}-e^{-2i\vfi_{\ell,0}}\right)\right[\hat\rho_\ell^{-4}q^{(1)}\left]\right\} +\mathcal{O}(\hat\alpha_{\ell,\lambda}k^{-3/2}+k^{-1/2}n_0^{-1/2})\\
&\hskip 48mm =\mathcal{O}(\hat\alpha_{\ell,\lambda}k^{-1}+k^{-1/2}n_0^{-1/2}) \label{eq:alpha11}
\end{align}
We can prove the analogues of (\ref{eq:alpha2}) and (\ref{eq:alpha3}) and similar
 bounds for $\Delta_{\half} \alpha_{\ell+\half, \lambda}$ the same way. We can also prove
\begin{equation}\label{eq:extra}
\ev\left(\Delta_{\half} e^{i\varphi_{\ell,\lambda}}-\Delta_{\half}  e^{i\varphi_{\ell,0}}|\mathcal{F}_{\ell}\right)\leq ck^{-1}\hat \alpha_\ell+cn_0^{-1/2}k^{-1/2}
\end{equation}
this is the analogue of Lemma 32 from \cite{carousel} and it can be proved exactly the same way.

To get (\ref{eq:alpha1}) we write
\begin{eqnarray}\nonumber
&&\ev\left[\Delta \alpha_{\ell,\lambda}\big| \vfi_{\ell, 0}=x, \vfi_{\ell,\lambda}=y \right]\\
\nonumber &&\hskip20mm =\ev\left[\Delta_{\half} \alpha_{\ell,\lambda}\big|  \vfi_{\ell, 0}=x, \vfi_{\ell,\lambda}=y\right]+\ev\left[E[\Delta_{\half} \alpha_{\ell+\half,\lambda}\big| \FF_{\ell+\half} ]\big| \vfi_{\ell, 0}=x, \vfi_{\ell,\lambda}=y \right],\\
\nonumber  &&\hskip20mm=\ev\left[\Delta_{\half} \alpha_{\ell,\lambda}\big|  \vfi_{\ell, 0}=x, \vfi_{\ell,\lambda}=y\right]+\ev\left[\Delta_{\half} \alpha_{\ell+\half,\lambda}\big| \vfi_{\ell+\half, 0}=x, \vfi_{\ell+\half,\lambda}=y \right]\\&&\hskip30mm+\mathcal{O}(\hat\alpha_{\ell,\lambda}k^{-2}+k^{-3/2}n_0^{-1/2})\nonumber
\end{eqnarray}
where the last line follows from (\ref{eq:extra}) and the just proved half step estimates. Now applying (\ref{eq:alpha11}) and the corresponding estimate for $\Delta_{\half} \alpha_{\ell+\half,\lambda}$  we get (\ref{eq:alpha1}). The other to estimates follow similarly.
\end{proof}

The next lemma provides a Gronwall-type estimate for the relative phase function. This will be the main ingredient in the proof of Proposition \ref{prop:middle}. The proof is based on the single step estimates of Proposition \ref{prop:alpha1step} and the oscillation estimates of Lemma \ref{l_sum}, the latter will be proved in the Appendix.
\begin{lemma}\label{lem:alphasum}
There exist constants $c_0,c_1, c_2$ depending on $\bar \lambda, \beta$ and a finite set $J$
 depending on $n, n_1, m_1$
so that with $y=n_0^{-1/2}(n_1^{1/6}\vee \log n_0)$
 we have
\[\left|\mathbb{E}\left(\alpha_{\ell_2,\lambda}-\alpha_{\ell_1,\lambda}|\mathcal{F}_{\ell_1}\right)\right|\leq c_1(y+\sqrt{\epsilon})+\mathbb{E}(\hat\alpha_{\ell_2-1}|\mathcal{F}_{\ell_1})/2+\sum_{\ell=\ell_1}^{\ell_2-2}b_\ell\mathbb{E}(\hat\alpha_\ell|\mathcal{F}_\ell)\]
\[0\le b_\ell\leq  c_0(n_1^{1/2} k^{-5/2}+k^{-3/2}+\max(n_1^{1/3}, 1)k^{-2}+1_{(\ell\in J)})\]
if $\mathcal{K}>c_2$, $|\lambda|\le \bar \lambda$ and $n_0(1-\eps)\le \ell_1\le \ell_2\le n_2$.
\end{lemma}
\begin{proof} Recall that $k=n_0-\ell$, $k_i=n_0-\ell_i$.
 We denote $x_\ell=\mathbb{E}[\hat\alpha_\ell|\mathcal{F}_{\ell_1}]$ and set
\begin{eqnarray*}e_{1,\ell}&=&\frac{1}{n_0}\left(e^{-i \vfi_{\ell,\lambda}}-e^{-i\vfi_{\ell,0}}\right)\left[(- v_{\lambda}- i q^{(1)}/2)\hat \rho_\ell^{-2}+(- \hat v_{0}- i q^{(2)}/2)\right],\\
e_{2,\ell}&=&-\frac{i}{n_0}( e^{- 2i\vfi_{\ell,\lambda}}- e^{- 2i\vfi_{\ell,0}}) [\hat \rho_\ell^{-4} q^{(1)}+q^{(2)}]/4.\end{eqnarray*}
From Proposition \ref{prop:alpha1step} we can write
\[\left|\mathbb{E}[\alpha_2-\alpha_1|\FF_{\ell_1}]\right|\leq \left|\sum_{\ell=\ell_1}^{\ell_2-1}\Re(\eta_\ell e_{1,\ell})\right|+\left|\sum_{\ell=\ell_1}^{\ell_2-1}\Re(\eta_\ell^2 e_{2,\ell})\right|+
c\sum_{\ell=\ell_1}^{\ell_2-1} x_{\ell}k^{-3/2}+c \sum_{\ell=\ell_1}^{\ell_2-1} k^{-1/2}n_0^{-1/2}\]
whose terms we denote $\zeta_1,\zeta_2, \zeta_3$ and $\zeta_4$ respectively.
Clearly, $\zeta_3$ is of the right form and
\[
|\zeta_4|\le \sum_{k=1}^{n_0 \eps} k^{-1/2} n_0^{-1/2}\le c \sqrt{\eps},
\]
so we only need to bound the first two terms.

We will use
\[\mathbb{E}\left(\Delta e^{i\varphi_{\ell,\lambda}}-\Delta e^{i\varphi_{\ell,0}}|\mathcal{F}_{\ell}\right)\leq ck^{-1}x_\ell+cn_0^{-1/2}k^{-1/2}\]
which is the `one-step' version of (\ref{eq:extra}) and can be proved the same way as Lemma 32 in \cite{carousel}. From this we get the estimates
\[|e_{i,\ell}|\leq cx_\ell/k,\qquad |\Delta e_{i,\ell}|\leq ck^{-2}x_\ell+cn_0^{-1/2}k^{-3/2}.\]
Then by Lemma \ref{l_sum} we have
\[|\zeta_1|\leq C x_{\ell_2-1} k_2^{-1} |F^{(1)}_{\ell_1,\ell_2-1}|+C
\sum_{\ell=\ell_1}^{\ell_2-2} |F_{\ell_1,\ell}^{(1)}|(x_\ell k^{-2}+n_0^{-1/2} k^{-3/2})\]
with $F_{\ell_1,\ell}\le C(n_1^{1/2} k^{-1/2}+1)$. Collecting the estimates and using $k_2\ge \mathcal{K} (n_1^{1/3}{\vee} 1)$ we get
\[
|\zeta_1|\leq c \mathcal{K}^{-1} x_{\ell_2-1}+c n_0^{-1/2} {\max(n_1^{1/6},1)}+
\sum_{\ell=\ell_1}^{\ell_2-2} x_\ell (n_1^{1/2} k^{-5/2}+k^{-2}).
\]
In order to bound $\zeta_2$ we use a similar strategy to the one applied in the proof of Lemma \ref{lem:bsum}. We divide the index set $[\ell_1, \ell_2]$ into finitely many intervals $I_1, I_2, \dots, I_a$  so that for each  $1\le j \le a$  one of the following three statements holds:
\begin{eqnarray}\label{divA}
&&  \textup{for each } \ell\in I_j \textup{ we have } k\ge \sqrt{n_1 m_1} \textup{ and } |\hat \rho_\ell^{2}\rho_\ell^2+1|\ge k^{-1/2},\\[4pt]\label{divB}
&&  \textup{for each } \ell\in I_j \textup{ we have } k\le \sqrt{n_1 m_1} \textup{ and } |\hat \rho_\ell^{2}\rho_\ell^2+1|\ge k^{-1/2},\\[4pt]\label{divC}
&&  \textup{for each } \ell\in I_j \textup{ we have } |\hat \rho_\ell^{2}\rho_\ell^2+1|\le k^{-1/2}.
\end{eqnarray}
It is clear that if we divide $[\ell_1, \ell_2]$ into parts at the $\sqrt{n_1 m_1}$ and the solutions of  $|\hat \rho_\ell^{2}\rho_\ell^2+1|=k^{-1/2}$ then the resulting partition will satisfy the previous conditions.
Since $|\hat \rho_\ell^{2}\rho_\ell^2+1|=k^{-1/2}$ has at most three roots (it is a cubic equation, see the proof of Lemma \ref{l_sum} for details) we can always get a suitable partition with at most five intervals. Moreover the endpoints of these intervals (apart from $\ell_1$ and $\ell_2$) will be the elements of a  set of size at most four with elements only depending on $n, m_1, n_1$.

We will estimate the sums corresponding to the various intervals $I_j$ separately. If $I_j$ satisfies condition (\ref{divC}) then we use
\[
|e_{2,\ell}|=\frac{1}{2\beta}|e^{- 2ix}- e^{- 2iy}| \frac{1}{k}|\hat \rho_\ell^{2}\rho_\ell^2+1|\le c k^{-3/2} x_\ell
\]
to get
\[
|\sum_{\ell\in I_j} \Re(\eta_\ell^2 e_{2,\ell})|\le c \sum_{\ell\in I_j} k^{-3/2} x_\ell
\]
If $I_j=[\ell_1^*, \ell_2^*]$ satisfies condition (\ref{divA}) then we use Lemma \ref{l_sum}
to get
\[|\sum_{\ell\in I_j} \Re(\eta_\ell^2 e_{2,\ell})|\leq c x_{\ell_1^*} (k_1^*)^{-1} |F^{(2)}_{\ell_1^*,\ell_2^*}|+c
\sum_{\ell={\ell_1^*}}^{{\ell_2^*}-1} |F^{(2)}_{{\ell+1},\ell_2^*}|(x_\ell k^{-2}+n_0^{-1/2} k^{-3/2})\]
where we have
\begin{equation}\label{mdl1}
|F^{(2)}_{{\ell+1}, \ell_2^*}|\le {|F^{(2)}_{\ell, \ell_2^*}|+1}\le c( k^{1/2}+n_1^{1/2} (k_2^*)^{-1/2}+1).
\end{equation}
We can bound the first term as
\[
x_{\ell_1^*} (k_1^*)^{-1} ( (k_1^*)^{1/2}+n_1^{1/2} (k_2^*)^{-1/2}+1)\le c \mathcal{K}^{-1/2} x_{\ell_1^*} \]
using $k\ge \max(n_1^{1/3},1)$.
For the general term in the sum we get
\begin{align}\nonumber
&c ( k^{1/2}+n_1^{1/2} (k_2^*)^{-1/2}+1)(x_\ell k^{-2}+n_0^{-1/2} k^{-3/2})\\&\hskip20mm\le c x_\ell(k^{-3/2}+\max(n_1^{1/3}, 1)k^{-2})+c( n_0^{-1/2} k^{-1}+\max(n_1^{1/3}, 1) n_0^{-1/2} k^{-3/2} )\label{err1}
\end{align}
Note that the sum of the error terms (\ref{err1}) is
\[
c \sum_{\ell={\ell_1^*}}^{{\ell_2^*}-1}  ( n_0^{-1/2} k^{-1}+\max(n_1^{1/3}, 1) n_0^{-1/2} k^{-3/2} )\le c( n_0^{-1/2} \log n_0+\max (n_1^{1/6},1) n_0^{-1/2})
\]
where we used $\max(n_1^{1/3}, 1)\le k \le n_0$. Putting our estimates together:
\begin{align}\label{middlebnd}
|\sum_{\ell\in I_j} \Re(\eta_\ell^2 e_{2,\ell})|\le& c \mathcal{K}^{-1/2} x_{\ell_1^*} +c \sum_{\ell=\ell_1}^{\ell_1^*-1} x_\ell(k^{-3/2}+\max(n_1^{1/3}, 1)k^{-2})\\&\hskip20mm+c( n_0^{-1/2} \log n_0+\max (n_1^{1/6},1) n_0^{-1/2}).\nonumber
\end{align}
The only case left is when $I_j=[\ell_1^*, \ell_2^*]$ satisfies condition  (\ref{divB}). If $\mu_n\ge \sqrt{m-n}$ then  we have the same estimate for $F^{(2)}_{\ell, \ell_2^*}$ as in (\ref{mdl1}) so we
get exactly the same bound as in (\ref{middlebnd}). If we have $\mu_n< \sqrt{m-n}$ then we use (\ref{partial1}) of Lemma (\ref{l_sum}) with the bound
\[
|F^{(2)}_{\ell_1^*, \ell}|\le c( k^{1/2}+n_1^{1/2} (k_1^*)^{-1/2}+1).
\]
Copying the previous arguments we get
\begin{eqnarray*}
|\sum_{\ell\in I_j} \Re(\eta_\ell^2 e_{2,\ell})|&\le& c \mathcal{K}^{-1/2} x_{\ell_2^*-1} +c \sum_{\ell={\ell_1^*}}^{{\ell_2^*}-1} x_\ell(k^{-3/2}+\max(n_1^{1/3}, 1)k^{-2})\\&&\hskip20mm+c( n_0^{-1/2} \log n_0+\max (n_1^{1/6},1) n_0^{-1/2}).
\end{eqnarray*}
Collecting our estimates, noting that $\ell_2^*-1$ is the endpoint of one of the intervals $I_j$
 and letting $\mathcal{K}$ be large enough we get the statement of the lemma.
\end{proof}

The proof of Proposition \ref{prop:middle} relies on the single step estimates of Proposition \ref{prop:alpha1step} and the following Gronwall-type lemma which was proved in \cite{carousel}.
\begin{lemma}\label{lem:turbogw}
Suppose that for positive numbers $x_\ell, b_\ell, c$, integers
$\ell_1<\ell\le \ell_2$  we
have
\begin{equation}\label{e_gr1}
x_{\ell}\le \frac{x_{\ell-1}}{2}+c+\sum_{j=\ell_1}^{\ell-1} b_j
x_j.
\end{equation}
Then $
x_{\ell_2}\le 2\,(x_{\ell_1}+c) \exp\left(3
\sum_{j=\ell_1}^{\ell_2-1} b_j\right)$.
\end{lemma}
Now we are ready to prove Proposition \ref{prop:middle}.
\begin{proof}[Proof of Proposition \ref{prop:middle}] We will adapt the proof of Proposition 28 from \cite{carousel}. Let $a=\alpha_{\ell_1,\lambda}$ and define $a_{\lozenge},a^{\lozenge}\in 2\pi\mathbb{Z}$ so that $[a_{\lozenge},a^{\lozenge})$ is an interval of length $2\pi$ containing $a$. We can assume that $\lambda\ge 0$, the other case being very similar. We will drop the index $\lambda$ from $\alpha$ and we will write $\mathbb{E}(.)=\mathbb{E}(.|\mathcal{F}_{\ell_1}).$ 

We will show that there exists $c_0$ so that if
$\mathcal{K}>c_0$, then if $\tilde a=a_\diamondsuit$ or $a^\diamondsuit$ then
\begin{eqnarray}\label{eq:tightness}
\ev |\alpha_{\ell_2}-\tilde a |& \le& c_1 (
|a-\tilde a| + \sqrt{\eps}+y).
\end{eqnarray}
The claim of the proposition follows from this by an application
of the triangle inequality, the additional
condition $\kappa>c_0$ is treated via the error term $1/\kappa$.

In order to prove (\ref{eq:tightness}) for $\tilde a=a_\diamondsuit$ we follow the steps described in Proposition 28 from \cite{carousel}. Using the exact same argument we
only need to prove that for the coefficients $b_\ell$ in Lemma \ref{lem:alphasum} are bounded 
by a constant depending only on $\bar \lambda, \beta$ and that $\alpha$ never goes below an integer multiple of $2\pi$ that it passes. 

The first statement is easy to check, we have,
\begin{align*}
\sum_{\ell=\ell_1}^{\ell_2-1}b_{\ell}\leq c_0\left(n_1^{1/2}\min(n_1^{-1/2},1)+\min(n_1^{-1/6},1)+\max(n_1^{1/3}, 1)\min(n_1^{-1/3}, 1)+\#{(J)}\right)<c^{\prime}.
\end{align*}
To prove the other statement first recall the evolution steps (\ref{eq:fievo}) and that $\alpha_{j,\lambda}=\vfi_{j,\lambda}-\vfi_{j,\lambda}$ for $j\in \ZZ/4$.
Using the fact that the maps ${\mbf{L}}_{\ell,\lambda}$, $\hat {\mbf{L}}_{\ell,\lambda}$ and their conjugates are monotone in $\lambda$ (as functions on $\R$) we get that  $\alpha_{\ell,\lambda}\leq \alpha_{\ell+\qt,\lambda}$ and $\alpha_{\ell+\half,\lambda}\leq \alpha_{\ell+\tqt,\lambda}.$ Moreover, since $\left({\mbf{S}}_{\ell,0}\right)^{\hat {\mbf{Q}}_{\ell}}$ and $\left(\hat {\mbf{S}}_{\ell,0}\right)^{{\mbf{Q}}_{\ell}}$ are $2\pi-$quasiperiodic functions on $\R$ we have $\lfloor\alpha_{\ell+\qt,\lambda}\rfloor_{2\pi}=\lfloor\alpha_{\ell+\half,\lambda}\rfloor_{2\pi}$ and $\lfloor\alpha_{\ell+\tqt,\lambda}\rfloor_{2\pi}=\lfloor\alpha_{\ell+1,\lambda}\rfloor_{2\pi}$. Hence, we get the following inequality:
\[\lfloor\alpha_{\ell,\lambda}\rfloor_{2\pi}\leq\lfloor\alpha_{\ell+\qt,\lambda}\rfloor_{2\pi}=
\lfloor\alpha_{\ell+\half,\lambda}\rfloor_{2\pi}\leq \lfloor\alpha_{\ell+\tqt,\lambda}\rfloor_{2\pi}=\lfloor\alpha_{\ell+1,\lambda}\rfloor_{2\pi},\] which implies that $\alpha$ never goes below an integer multiple of $2\pi$ that it passes. This means $\alpha_\ell\ge a_\diamondsuit$ and $\alpha_\ell-a_\diamondsuit\ge \hat \alpha_\ell$  for $\ell\ge \ell_1$.
 
Lemma \ref{lem:alphasum}  provides the bound
 \begin{eqnarray*}
|\ev \alpha_\ell- a_\diamondsuit|&\le& (a-a_\diamondsuit)+c
(y+\sqrt{\eps})+\ev \hat \alpha_{\ell-1}/2+\sum_{j=\ell_1}^{\ell-2} b_j
\ev \hat \alpha_j.
 \end{eqnarray*}
Since $|\ev \alpha_\ell- a_\diamondsuit|\ge \ev \hat \alpha_\ell$,  inequality (\ref{eq:tightness}) follows for $\tilde  a=a_{\diamondsuit}$ via  Lemma \ref{lem:turbogw}.

In order to deal with the $\tilde a=a^{\diamondsuit}$ case in (\ref{eq:tightness}) we define $T\in \ZZ/2$ the first time when $\alpha_T\ge a^\diamondsuit$.
Note that $\alpha$ can only pass an integer multiple of $2\pi$ in the $\ell\to \ell+\qt$ or $\ell+\half\to \ell+\tqt$ steps, and $\vfi$ evolves deterministically in these steps. This means that  $T-\half$ is a stopping time with respect to $\FF_j$, $j\in \ZZ/2$. 

For large enough $\mathcal{K}$ we have $\frac{\bar \lambda}{4
\sqrt{n_0 k}}\le \frac1{10}$. Then by Lemma \ref{lem:ash} we get the uniform bound
\[
\alpha_{j+\qt,\lambda}-\alpha_{j,\lambda}\le c n_0^{-1/2}, \quad  j\le \ell_2, j \in \ZZ/2.
\]
By the strong Markov property and the bound (\ref{eq:alpha1}) we get
\[
\ev[(\alpha_T-a^{\diamondsuit}) \mathbf{1}(T\le \ell_2)]\le c n_0^{-1/2} \quad \textup{and} \quad\ev[(\alpha_{T+\half}-a^{\diamondsuit}) \mathbf{1}(T\le \ell_2)]\le c n_0^{-1/2}.
\]
Using this together with the first part of the proof and the strong Markov property again we get
\begin{eqnarray}
\ev (\alpha_{\ell_2}-a^\diamondsuit)^{+}
  \nonumber
\nonumber
 &=&
\ev \left[ \mathbf{1}{(T\le \ell_2)}\ev\left[\alpha_{\ell_2}-a^\diamondsuit\big | \mathcal F_T  \right]\right]\\
 \nonumber
 &\le&
 c_1 (\ev \left[(\alpha_{T}-a^\diamondsuit) \mathbf{1}{(T\le \ell_2)} \right]\ + \sqrt{\eps}+ y)\\
&\le&c_1' (\sqrt{\eps}+y), \label{e_tightness2}
\end{eqnarray}
Lemma \ref{lem:alphasum} gives
\[
|\ev \alpha_{\ell}-a^\diamondsuit | \le (a^\diamondsuit-a)+c
(y+\sqrt{\eps})+\ev \hat \alpha_{\ell-1}/2+\sum_{j=\ell_1}^{\ell-2} b_j
\ev \hat \alpha_j.
\]
Then by (\ref{e_tightness2}) and the identity $|a|=-a+2
a^+$ we get
\begin{eqnarray*}
\ev |\alpha_{\ell}-a^\diamondsuit |&\le& |\ev \alpha_{\ell}-a^\diamondsuit| +2 \ev (\alpha_{\ell}-a^\diamondsuit)^+\\
&\le& (a^\diamondsuit-a)+c (y+\sqrt{\eps})+\ev
\hat \alpha_{\ell-1}/2+\sum_{j=\ell_1}^{\ell-2} b_j \ev \hat \alpha_j.
\end{eqnarray*}
Since $\hat \alpha_{\ell}\le |\alpha_\ell-a^\diamondsuit|$,  the Gronwall-type estimate in
Lemma \ref{lem:turbogw} implies (\ref{eq:tightness}) with $\tilde a=a^{\diamondsuit}$.
\end{proof}

\subsection{The uniform limit}

\begin{proposition}\label{prop:unifmiddle}
Assume that $\mathcal{K}=\mathcal{K}_n$ with $\mathcal{K}\to \infty$ and that $n_0^{-1}\mathcal{K}(n_1^{1/3}\vee 1)\to 0$. Then, $\varphi_{n_2,0}$ modulo $2\pi$ converges in distribution to the uniform distribution on $(0,2\pi).$
\end{proposition}
\begin{proof}
We can use exactly the same argument as in Proposition 33, \cite{carousel}. We show that given $\epsilon>0,$ every subsequence has a further subsequence along which $\varphi_{n_2,0}$ modulo $2\pi$ is eventually $\epsilon$-close to the uniform distribution. 
We set $\xi=\lfloor\mathcal{K}(n_1^{1/3}\vee 1)\rfloor$ and pick $\tau=\tau(\epsilon)$ with $\tau \xi \le n_2$. Because of $n_0^{-1}\mathcal{K}(n_1^{1/3}\vee 1)\to 0$ we will be able to let  $\tau\to \infty$.
We will show that for any fixed $\tau$ the distribution of $\varphi_{n_2,0}-\varphi_{n_2-\tau\xi,0}$ given $\mathcal{F}_{n_2-\tau\xi}$ is asymptotically normal with a variance going to $\infty$ as $\tau\to \infty$. From that the statement will follow.

Note that the arguments of Proposition \ref{prop:phisde} can be repeated for the evolution of $\vfi_{n_2-\tau \xi+\ell,0}$ with $0\le \ell\le \tau \xi$ which gives that $\vfi_{n_2,0}$ conditioned on $\FF_{n_2-\tau\xi}$ converges to 
$\vfi_0(1-(1+\tau)^{-1})$ where $\vfi_{0}(t)$ is the solution of (\ref{eq:sdelimfi}) with $\lambda=0$. This is just a normal random variable, its variance is given by the integrating the sum of the squares of the independent diffusion coefficients on $[0,1-(1+\tau)^{-1}]$. This is at least as big  as the variance coming from the $dW$ term which gives
$\int_{0}^{1-(1+\tau)^{-1}}\frac{2}{\beta(1-t)}dt$. This goes to $\infty$ if $\tau\to \infty$ as required.
\end{proof}

\section{Last stretch}\label{s:last}
The purpose of this section is to prove that on the interval $[n_2,n]$ the relative target phase 
function $\alpha^{\odot}_{\ell,\lambda}=\vfi^{\odot}_{\ell,\lambda}-\vfi_{\ell, 0}^{\odot}$ does not change much.
\begin{proposition}\label{prop:end}
For any fixed $\lambda\in \R$ and $\mathcal{K}>0$ we have $\alpha^{\odot}_{n_2,\lambda}\cp 0$.
\end{proposition}
The length of the interval $[n_2,n]$ is equal to $n_1+\mathcal{K}(n_1^{1/3}\vee 1)$, up to an error of order 1. By taking an appropriate subsequence of $n$ (see Remark \ref{rem:subseq}) we may assume that $n_1$ has a finite or infinite limit. We will consider these two cases separately.

\begin{proof}[Proof of Proposition \ref{prop:end} in the $\lim n_1<\infty$ case] By (\ref{eq:n0n1}) we have that $|m-n-\mu_n^2|/\mu_n$ converges, by taking an appropriate subsequence we can assume that the limit also exists without the absolute values.
Note that condition (\ref{eq:mu_cond}) implies that $\lim m/n>1$ and thus $m-n\to \infty$.

We may also assume that $n- n_2$ is eventually equal to an integer $\xi$. From there we proceed similarly as in \cite{carousel}. We first note that by (\ref{eq:target1}) and (\ref{eq:target2})
we have 
\[
\vfi^{\odot}_{n-\xi,\lambda}\lstar \mbf{Q}_{n-\xi-1}^{-1}=0\lstar \mbf{R}_{n-1,\lambda}\dots \mbf{R}_{n-\xi,\lambda} \mathbf{T}_{n-\xi}
\]
where $ \mbf{R}_{\ell,\lambda} =\hat {\mbf{M}}_\ell^{-1} \hat{\mbf{J}}_\ell^{-1}  {\mbf{M}}_\ell^{-1}  {\mbf{J}}_\ell^{-1}$. From (\ref{eq:recrr}) and (\ref{eq:recr}) we get
\[
 r\ldot \mbf{R}_{n-j,\lambda} =\frac{r s_{n-j}^2+r s_{n-j} Y_{n-j}-s_{n-j} \Lambda }{p_{n-j} p_{n-j+1}-\Lambda ^2+p_{n-j+1} X_{n-j}+r s_{n-j} \Lambda +r Y_{n-j} \Lambda }
\]
where $\Lambda=\mu_n+\lambda/(4 n_0^{1/2})$. Using (\ref{eq:n0n1}),  $s_{n-j}=\sqrt{j-1/2}$, $p_{n-j}=\sqrt{m-n+j-1/2}$ and $m-n\to \infty$ we get that $r\ldot \mbf{R}_{n-j,\lambda} -r\ldot \mbf{R}_{n-j,0} \cp 0$ for any fixed $j$. This leads to  
\[
\alpha^{\odot}_{n_2,\lambda}=\vfi^{\odot}_{n-\xi,\lambda}\lstar \mbf{Q}_{n-\xi-1}^{-1}-\vfi^{\odot}_{n-\xi,0}\lstar \mbf{Q}_{n-\xi-1}^{-1}\to 0.\qedhere
\]
\end{proof}
If $\lim n_1=\infty$ then we will need the edge scaling results proved in \cite{RRV} which are summarized in Theorem \ref{thm:softedge} of the introduction. 
The initial condition
$f(0)=0, f'(0)=1$ for the operator $\mathcal{H}_\beta$ in the theorem comes from the fact that the discrete eigenvalue
equation for $A_{n,m} A_{n,m}^T$ with
an eigenvalue $\Lambda=(\sqrt{n}+ \sqrt{m})^2+\frac{(\sqrt{m}+ \sqrt{n})^{4/3}}{(mn)^{1/6}} \nu$ is
equivalent to a three-term recursion for the vector entries
$w_{\ell,\Lambda}$ (c.f. (\ref{ev}))
with the initial condition  $w_{0,\nu}=0$ and $w_{1,\nu}\neq 0$.

By \cite{RRV}, Remark 3.8, the results of \cite{RRV} extend to solutions of the
same three-term recursion with  more general initial conditions.
We say that a value of $\nu$ is an eigenvalue for a family of
recursions parameterized by $\nu$ if the corresponding recursion
reaches $0$ in its last step.
%
Suppose that for given $\zeta\in
[-\infty,\infty]$ the initial condition for the three-term recursion equation
satisfies
\begin{equation}\label{eq:initial}
\frac{(mn)^{-1/3}}{(\sqrt{m}+\sqrt{n})^{-2/3}}\frac{w_{0,\nu}}{
(w_{1,\nu}-w_{0,\nu})}=\frac{(mn)^{-1/3}}{(\sqrt{m}+\sqrt{n})^{-2/3}} (r_{0,\nu}-1)^{-1} \cp \zeta,
\end{equation}
uniformly in $\nu$ with $r_0:=w_{1,\nu}/w_{0,\nu}$. Here
the factor $\frac{(mn)^{1/3}}{(\sqrt{m}+\sqrt{n})^{2/3}}$ is the spatial scaling for the
problem (\cite{RRV}, Section 5).   Then the eigenvalues of this family of recursions
converge to those of the stochastic Airy operator with initial
condition $f(0)/f'(0)=\zeta$. The corresponding point process
$\Xi_\zeta$ is also a.s.~simple and it will satisfy the following non-atomic property: for any $x\in \R$ we have $\pr(x\in \Xi_\zeta)=0$ (see \cite{RRV}, Remark
3.8). Similar statement holds at the lower edge if $\liminf m/n>1$ with $(r_{n,\nu}+1)^{-1}$ in (\ref{eq:initial}). (In this case one first multiplies the off-diagonal entries of  $A_{n,m} A_{n,m}^T$ by $-1$ before applying the arguments of \cite{RRV}, this will not change the eigenvalues.)

If $m/n\to \gamma\in [1,\infty)$ then we can rewrite (\ref{eq:initial}) jointly for the upper and lower soft edge as
\begin{equation}
\frac{\gamma^{-1/3}}{(\sqrt{\gamma}\pm 1)^{-2/3}} n^{-1/3}(r_{0,\nu}\mp 1)^{-1} \cp \zeta,\label{eq:initial1}.
\end{equation}
The multiplier is 1 in the case $\gamma=\infty$ and it is 
always a finite nonzero value unless $\gamma=1$ and we are at the lower edge.

\begin{proof}[Proof of Proposition \ref{prop:end} if $\lim n_1=\infty$]

By taking an appropriate subsequence, we may assume that $\mu_n-\sqrt{m-n}$ is always positive or always negative.
According to the proof of Lemma 34 in \cite{carousel} we need to consider  the family of recursions 
\[
r_{\ell+1, \nu}=r_{\ell,\nu} \ldot{\mbf{J}}_\ell {\mbf{M}}_\ell  \hat{\mbf{J}}_\ell \hat {\mbf{M}}_\ell, \quad n_2\le \ell\le n
\]
with initial condition 
\[
r_{n_2,\nu}=x\ldot \mbf{T}_{n_2}^{-1}=\Im(\rho_{n_2}) x+\Re(\rho_{n_2})
\]
for a given $x\in \R$ and show that the probability of having an eigenvalue in $[0,\lambda]$ converges to 0 as $n\to \infty$. 

We introduce
\[
n^*=n-n_2, \quad m^*=m-n_2.
\]
Note that the recursion is determined by the bottom $(2n^*)\times(2  n^*)$ submatrix of $\tilde A_{n,m}^{D^{(n)}}$ in (\ref{eq:tridag}) which has the exact same  distribution as the matrix $\tilde A_{n^*,  m^*}^{D^{(n^*)}}$. Thus we can consider the eigenvalue equation for $\tilde A_{n^*,  m^*}^{D^{(n^*)}}$ with generalized initial condition 
\[
r^*_{0,\lambda}=\Im(\rho_{n_2}) x+\Re(\rho_{n_2}).
\]
(We will use $*$ to denote that we are working with the smaller matrix.)
We will transform this back to an eigenvalue equation for $\tilde A_{n^*,  m^*}$ and then $A_{ n^*, m^*} A_{ n^*, m^*}^T$ with a generalized initial condition.

Recall the recursion  (\ref{eq:recrr}) and  that $r^*_\ell=u^*_{\ell+1}/v^*_{\ell}$,  $\hat r^*_{\ell}=v^*_{\ell}/u^*_{\ell}$. From this we get
\begin{eqnarray}\nonumber
\frac{v^*_1}{v^*_0}&=&r^*_0 \hat r^*_1=r^*_0 \left(-\frac1{r^*_0}\cdot \frac{s^*_0}{p^*_0} +\frac{\Lambda}{p^*_0}\right)\left(1+\frac{X^*_0}{p^*_0}\right)^{-1}\\
&=&\left(- \frac{s^*_0}{p^*_0} +(\Im(\rho_{n_2}) x+\Re(\rho_{n_2}))\frac{\Lambda}{p^*_0}\right)\frac{\beta p_1^* p_0^*}{\tilde \chi^2_{\beta(m^*-1)}}\label{eq:init0}
\end{eqnarray}
where $\Lambda=\mu_n+\frac{\lambda}{4\sqrt{n_0}}$. Denote the solution of the generalized eigenvalue equation for $\tilde A_{n^*,  m^*}$  corresponding to $\Lambda$ with $\tilde u_1^*, \tilde v_1^*, \dots$. If we remove the conjugation with $D^{(n^*)}$ from $\tilde A_{n^*,  m^*}^{D^{(n^*)}}$ then from (\ref{eq:init0}) we get
\begin{equation}\label{eq:init1}
\frac{\tilde v^*_1}{\tilde v^*_0}=\left(- \frac{s^*_0}{p^*_0} +(\Im(\rho_{n_2}) x+\Re(\rho_{n_2}))\frac{\Lambda}{p^*_0}\right) \frac{\chi_{\beta(n^*-1)} p_0^*}{ \tilde \chi_{\beta(m^*-1)}  s_1^*}.
\end{equation}
By Remark \ref{rem:sing} this is exactly the initial condition for the generalized eigenvalue equation for 
$A_{ n^*, m^*} A_{ n^*, m^*}^T$ with eigenvalue $\Lambda^2=\mu_n^2+\frac{\mu_n \lambda}{2\sqrt{n_0}}+\frac{\lambda^2}{16 n_0}$.

The spectrum of the matrix $A_{ n^*, m^*} A_{ n^*, m^*}^T$ is concentrated asymptotically to $[(\sqrt{m^*}-\sqrt{n^*})^2,(\sqrt{m^*}+\sqrt{n^*})^2]$. A direct computation shows that
$
\sqrt{m_1}\pm \sqrt{n_1}=\mu_n,
$
where we have $+$ if $\mu_n-\sqrt{m-n}>0$ and $-$ otherwise. This means that if $\mu_n<\sqrt{m-n}$ then our original bulk scaling around $\mu_n$ corresponds to the lower edge scaling of  $A_{ n^*, m^*} A_{ n^*, m^*}^T$ and if $\mu_n> \sqrt{m-n}$ then we get the upper edge scaling. (Note that because of our assumptions if $\mu_n<\sqrt{m-n}$ than $\liminf m/n>1$.)

Again, by taking an appropriate subsequence, we may assume that $m^*/n^*\to \gamma^*\in[1,\infty]$.
We first check that the initial condition (\ref{eq:init1}) satisfies (\ref{eq:initial1}), this is equivalent to showing that
$
(n^*)^{1/3} \left( \frac{\tilde v^*_1}{\tilde v^*_0}\mp 1\right)
$
converges in probability to a constant. Since $n_1\to \infty$ we have
\begin{align*}
&m^*>n^*\to \infty, \qquad \rho_{n_2}=\pm 1+i \sqrt{\mathcal{K}} (n^*)^{-1/3}+o( (n^*)^{-1/3}),\\
&\hskip10mmn_1=n^*-\mathcal{K} (n^*)^{1/3}, \qquad m_1=m^*-\mathcal{K} (n^*)^{1/3},\\
&\mu_n=\sqrt{m^*}\pm\sqrt{n^*}-\frac{\mathcal{K}}{2} (n^*)^{1/3}\left((m^*)^{-1/2}\pm(n^*)^{-1/2}\right)+o((n^*)^{-2/3}).
\end{align*}
Since $\ell^{-1/2}{\chi_\ell}\cp 1$ as $\ell\to \infty$ we get that in probability
\[
\lim \, (n^*)^{1/3} \left( \frac{\tilde v^*_1}{\tilde v^*_0}\mp 1\right)=\sqrt{\mathcal{K}} x \left(1 \pm (\gamma^*)^{-1/2}\right)
\]
and the convergence is uniform if we assume that $\lambda$ is bounded. This means that we may apply the edge scaling result, we just need to convert the scaling from the bulk to the edge. In the bulk we were interested in the interval $I=[\mu_n^2, (\mu_n+\lambda n_0^{-1/2}/ 4)^2]$. If we apply the edge scaling, then this interval becomes $\frac{(m^* n^*)^{1/6}}{(\sqrt{m^*}\pm\sqrt{n^*})^{4/3}} (I-(\sqrt{m^*}\pm \sqrt{n^*})^2)$. Using our asymptotics for $\mu_n$ we get that for the end point of the interval 
\[
\frac{(m^* n^*)^{1/6}}{(\sqrt{m^*}\pm\sqrt{n^*})^{4/3}} (\mu_n^2-(\sqrt{m^*}\pm \sqrt{n^*})^2)\to-\frac{\mathcal{K}}{2} (1\pm (\gamma^*)^{-1/2})^{1/6},
\]
and for the length we get
\[
\lim \left((\mu_n+\lambda n_0^{-1/2}/ 4)^2-\mu_n^2\right)\frac{(m^* n^*)^{1/6}}{(\sqrt{m^*}\pm\sqrt{n^*})^{4/3}} =\frac{(\gamma^*)^{1/6}}{(\sqrt{\gamma^*}\pm 1)^{1/3}} \lim\frac{\lambda}{\sqrt{n_0}}=0.
\]
(We used that if $\gamma^*=1$ then we must have $+$ in the $\pm$.)
This means that after the edge rescaling the interval shrinks to a point, meaning that the probability that 
our original recursion has an eigenvalue in $[0,\lambda]$ converges to the probability that the limiting edge point process has a point at a given value  which is equal to 0. 
\end{proof}

\section{Appendix}
In this section we provide the needed oscillation estimates.
\begin{lemma}\label{l_vdcorp}
Suppose that $2\pi>\theta_1>\theta_2>\ldots>\theta_m>0$ and
let $s_\ell=\sum_{j=1}^\ell \theta_j$. Then
\[
|\sum_{j=1}^m e^{i s_j}|\le
c(\theta_m^{-1}+(2\pi-\theta_1)^{-1})\le c'(|e^{i\theta_m/2}-1|^{-1}+|e^{i\theta_1/2}+1|^{-1}).
\]
\end{lemma}
\begin{proof}
The first inequality is the same as Lemma 36  from \cite{carousel} and the second inequality is straightforward.
\end{proof}

\begin{lemma}\label{l_sum}
Let $1\le \ell_1<\ell_2\le n_0$ and $F^{(i)}_{\ell_1,\ell_2}=\sum_{j=\ell_1}^{\ell_2} \eta_j^i$ for $i=1,2$ and $g_\ell\in \CC\,$. Then for $i=1, 2$:
\begin{align}\label{partial1}
&\left|\Re\sum_{\ell=\ell_1}^{\ell_2}g_\ell \eta_{\ell}^i\right|\leq  |F^{(i)}_{\ell_1,\ell_2}| |g_{{\ell_2}}|
+\sum_{\ell=\ell_1}^{\ell_2-1}|F^{(i)}_{\ell_1,\ell}||g_{\ell+1}-g_\ell|,\\
&\left|\Re\sum_{\ell=\ell_1}^{\ell_2}g_\ell \eta_{\ell}^i\right|\leq  |F^{(i)}_{\ell_1,\ell_2}|g_{{\ell_1}}||
+\sum_{\ell=\ell_1}^{\ell_2-1}|F^{(i)}_{{\ell+1},\ell_2}||g_{\ell+1}-g_\ell|. \label{partial2}
\end{align}
We also have the following estimates:
\begin{eqnarray}\label{F1}
|F^{(1)}_{\ell_1,\ell_2}|&\le& c(1+n_1^{1/2} k_2^{-1/2})
\end{eqnarray}
\begin{eqnarray}\label{F2}
|F^{(2)}_{\ell_1,\ell_2}|&\le& \left\{ \begin{array}{ll}
c\left(|\rho_{\ell_1}^2 \hat \rho_{\ell_1}^2+1|^{-1} +1+n_1^{1/2} k_2^{-1/2}\right)&\textup{if ($\star$) or ($\star\star$)},\\[5pt]
c\left(|\rho_{\ell_2}^2 \hat \rho_{\ell_2}^2+1|^{-1} +1+n_1^{1/2} k_1^{-1/2}\right)&\textup{if ($\star\star\star$)}
\end{array}\right.
\end{eqnarray}
where the conditions are given by
\begin{eqnarray}\nonumber
(\star)&:& \mu_n\ge \sqrt{m-n} \textup{ with } k_0, k_1\ge \sqrt{m_1 n_1} \,\textup{ or }\,  k_0, k_1\le \sqrt{m_1 n_1},\\
(\star\star)&:& \mu_n< \sqrt{m-n} \textup{ with } k_0, k_1\ge \sqrt{m_1 n_1},\\
(\star\star\star)&:& \mu_n< \sqrt{m-n} \textup{ with } k_0, k_1\le \sqrt{m_1 n_1}.\nonumber
\end{eqnarray}
\end{lemma}
\begin{proof}
The bounds (\ref{partial1}) and (\ref{partial2}) follow from partial summation. In order to prove the bounds on $F^{(i)}_{\ell_1, \ell_2}$
we will apply Lemma \ref{l_vdcorp}, but we need to consider various cases. Note that the constant $c$ might change from line to line.\\
\noindent \textbf{Case 1: $\mu_n\ge \sqrt{m-n} , \Re\rho_\ell\ge0$}.\\
We have the bounds
\begin{eqnarray*}
k^{1/2} (n_1+k)^{-1/2}\le \arg \rho_{\ell}, \quad n_1^{1/2} (n_1+k)^{-1/2}\le \pi/2-\arg \rho_{\ell},\\
k^{1/2} (m_1+k)^{-1/2}\le \arg \hat\rho_{\ell}, \quad m_1^{1/2} (m_1+k)^{-1/2}\le \pi/2-\arg \hat\rho_{\ell}.
\end{eqnarray*}
and $\arg(\rho_\ell \hat \rho_\ell)$ is decreasing.
The sequence $\arg (\hat \rho_\ell^2 \rho_\ell^2)$ satisfies the conditions of Lemma \ref{l_vdcorp}
and using $m_1>n_1$ and $m_1>c n>c n_0$ we get the bound
 \[
|\sum_{\ell=\ell_1}^{\ell_2} \eta_\ell|\le c\left((n_1+{k_2})^{1/2} k_2^{-1/2}+(k_1+{m_1})^{1/2} m_1^{-1/2}\right)\le c(n_1^{1/2} k_2^{-1/2}+1).
\]
We have
\begin{equation}\nonumber
\rho_\ell \hat \rho_\ell=\frac{\sqrt{m_1 n_1}-k}{\sqrt{k+m_1} \sqrt{k+n_1}}+\frac{i \sqrt{k} \left(\sqrt{m_1}+\sqrt{n_1}\right)}{\sqrt{k+m_1} \sqrt{k+n_1}}
\end{equation}
which means that if $k_1,k_2\ge \sqrt{m_1n_1}$ or $k_1,k_2\le \sqrt{m_1 n_1}$ then we can use Lemma \ref{l_vdcorp} for the sequence $\arg (\hat \rho_\ell^4 \rho_\ell^4)$. (In the first case $\pi/2\le \arg (\hat \rho_\ell \rho_\ell)< \pi$ while in the second case $0< \arg (\hat \rho_\ell \rho_\ell)\le \pi/2$.) From the lemma we get
\[
|\sum_{\ell=\ell_1}^{\ell_2} \eta_\ell^{{{2}}}|\le c\left(|\rho_{\ell_1}^2 \hat \rho_{\ell_1}^2+1|^{-1} +|\rho_{\ell_2}^2 \hat \rho_{\ell_2}^2-1|^{-1}\right)
\]
and explicit computation together with $m_1>c n>c k$  gives
\[
|\rho_\ell^2 \hat \rho_\ell^2-1|^{-1}=\frac{\sqrt{k+m_1} \sqrt{k+n_1}}{2 \sqrt{k} \left(\sqrt{m_1}+\sqrt{n_1}\right)}\le c(1+n_1^{1/2} k^{-1/2}).
\]
This finishes the proof of the lemma in this case.\\
\noindent \textbf{Case 2: $\mu_n< \sqrt{m-n} , \Re\rho_\ell<0$}.\\ Now we have
\begin{equation}\label{roro2}
\rho_\ell \hat \rho_\ell=-\frac{k+\sqrt{m_1} \sqrt{n_1}}{\sqrt{k+m_1} \sqrt{k+n_1}}+\frac{i \sqrt{k} \left(\sqrt{m_1}-\sqrt{n_1}\right)}{\sqrt{k+m_1} \sqrt{k+n_1}}.
\end{equation}
meaning $\pi/2< \arg (\hat \rho_\ell \rho_\ell)< \pi$. Differentiation of the real part shows that $\arg\rho_\ell \hat \rho_\ell$ decreases if $k>\sqrt{m_1 n_1}$ and
then it increases. This means that we can apply Lemma \ref{l_vdcorp} for the sequences  $\arg (\hat \rho_\ell^2 \rho_\ell^2)$ or  $\arg (\hat \rho_\ell^4 \rho_\ell^4)$ if $k_1, k_2\ge \sqrt{n_1 m_1}$ and the reversed versions of these sequences if  $k_1, k_2\le \sqrt{n_1 m_1}$.
From (\ref{roro2}) we get
\begin{equation}\label{roro2bnd}
\pi\le \arg \rho_\ell^2 \hat \rho_\ell^2, \qquad 2\pi-\arg \rho_\ell^2 \hat \rho_\ell^2\ge  \frac{2 \sqrt{k} \left(\sqrt{m_1}-\sqrt{n_1}\right)}{\sqrt{k+m_1} \sqrt{k+n_1}} \ge c (k^{-1/2} n_1^{1/2}+1)^{-1}
\end{equation}
where we  used $\sqrt{m_1-n_1}\ge c \sqrt{m_1}$ which follows from $m_1>c n>c n_1$.
Applying Lemma \ref{l_vdcorp} to $\arg (\hat \rho_\ell^2 \rho_\ell^2)$ (or the reversed sequence) we get
\[
|\sum_{\ell=\ell_1}^{\ell_2} \eta_\ell|\le
\left\{
\begin{array}{ll}
c(n_1^{1/2} k_{{1}}^{-1/2}+1)&\textup{if } k_1, k_2\ge \sqrt{n_1 m_1},\\
c(n_1^{1/2} k_{{2}}^{-1/2}+1)&\textup{if } k_1, k_2\le \sqrt{n_1 m_1}.
\end{array}
\right.
\]
Noting that $k_1^{-1/2}\le k_2^{-1/2}$ this proves (\ref{F1}) in Case 2 if $k_1, k_2\le \sqrt{n_1m_1}$ or $k_1, k_2\ge \sqrt{n_1m_1}$. If $k_1>\sqrt{n_1m_1}$ and $k_{{2}}<\sqrt{n_1m_1}$ then we can cut the sum in to parts at $\sqrt{n_1m_1}$ and since  for $k\ge \sqrt{m_1 n_1}$ we have $n_1^{1/2} k^{-1/2} \le 1$ we have (\ref{F1}) in this case as well.

To prove (\ref{F2}) we  apply Lemma \ref{l_vdcorp} to $\arg (\hat \rho_\ell^4 \rho_\ell^4)$ (or its reversed) to get
\[
|\sum_{\ell=\ell_1}^{\ell_2} \eta_\ell^2|\le \left\{ \begin{array}{ll}
c\left(|\rho_{\ell_1}^2 \hat \rho_{\ell_1}^2+1|^{-1} +|\rho_{\ell_2}^2 \hat \rho_{\ell_2}^2-1|^{-1}\right)&\textup{if }k_1, k_2\ge \sqrt{n_1 m_1},\\[5pt]
c\left(|\rho_{\ell_2}^2 \hat \rho_{\ell_2}^2+1|^{-1} +|\rho_{\ell_1}^2 \hat \rho_{\ell_1}^2-1|^{-1}\right)&\textup{if } k_1, k_2 \le \sqrt{n_1 m_1}.
\end{array}\right.
\]
From this (\ref{F2}) follows by noting that
\[
|\rho_\ell^2 \hat \rho_\ell^2-1|^{-1}=\frac{\sqrt{k+m_1} \sqrt{k+n_1}}{2 \sqrt{k} \left(\sqrt{m_1}-\sqrt{n_1}\right)}\le c(n_1^{1/2} k^{-1/2}+1).
\]
where the first equality is explicit computation and the inequality is from (\ref{roro2bnd}).
\end{proof}


\begin{thebibliography}{10}

\bibitem{DE}
Ioana Dumitriu and Alan Edelman.
\newblock Matrix models for beta ensembles.
\newblock {\em J. Math. Phys.}, 43(11):5830--5847, 2002.

\bibitem{DF}
Iona Dumitriu and Peter Forrester.
\newblock Tridiagonal realization of the anti-symmetric {G}aussian
  $\beta$-ensemble, 2009.
\newblock arXiv:0904.2216.

\bibitem{ESYY_wishart}
L\'aszl\'o Erd\H{o}s, Benjamin Schlein, Horng-Tzer Yau, and Jun Yin.
\newblock The local relaxation flow approach to universality of the local
  statistics for random matrices, 2009.
\newblock arXiv:0911.3687.

\bibitem{EthierKurtz}
Stewart~N. Ethier and Thomas~G. Kurtz.
\newblock {\em Markov processes}.
\newblock John Wiley \& Sons Inc., New York, 1986.

\bibitem{ForBook}
Peter Forrester.
\newblock {\em Log-gases and Random matrices}.
\newblock Princeton University Press, 2010.


\bibitem{KVV}
Eugene Kritchevski, Benedek Valk\'o, and B\'alint Vir\'ag.
\newblock The scaling limit of the critical one-dimensional random
  {S}chr\"{o}dinger operator, 2010.
\newblock preprint.

\bibitem{MP}
V.~A. Mar{\v{c}}enko and L.~A. Pastur.
\newblock Distribution of eigenvalues in certain sets of random matrices.
\newblock {\em Mat. Sb. (N.S.)}, 72 (114):507--536, 1967.

\bibitem{mehta}
Madan~Lal Mehta.
\newblock {\em Random matrices}.
\newblock Elsevier/Academic Press, 2004.

\bibitem{RRV}
Jose Ram{\'\i}rez, Brian Rider, and B\'alint Vir\'ag.
\newblock Beta ensembles, stochastic {A}iry spectrum, and a diffusion, 2007.
\newblock math/0607331.

\bibitem{RR}
Jose A.~Ram\'\i rez and Brian Rider.
\newblock Diffusion at the random matrix hard edge.
\newblock {\em Comm.  Math. Phys.}, (288):887--906, 2009.

\bibitem{SV}
Daniel~W. Stroock and S.~R.~Srinivasa Varadhan.
\newblock {\em Multidimensional diffusion processes}.
\newblock Classics in Mathematics. Springer-Verlag, Berlin, 1979.

\bibitem{TV_wishart}
Terence Tao and Van Vu.
\newblock Random covariance matrices: {U}niversality of local statistics of
  eigenvalues, 2009.
\newblock arXiv:0912.0966.

\bibitem{carousel}
Benedek Valk{\'o} and B{\'a}lint Vir{\'a}g.
\newblock Continuum limits of random matrices and the {B}rownian carousel.
\newblock {\em Invent. Math.}, 177(3):463--508, 2009.

\bibitem{VV3}
Benedek Valk\'o and B\'alint Vir\'ag.
\newblock Random {S}chr\"{o}dinger operators on long boxes, noise explosion and
  the GOE, 2009.
\newblock arXiv:0912.0097.

\bibitem{Wishart}
J.~Wishart.
\newblock The generalized product moment distribution in samples from a normal
  multivariate population.
\newblock {\em Biometrika A}, 20A:32--52, 1928.

\end{thebibliography}
\end{document}